\documentclass[12pt]{article}
\usepackage{ulem}
\usepackage{forest}
\usepackage[colorlinks, citecolor=red]{hyperref}
\usepackage{graphicx,amsmath,amsfonts,amssymb,color,mathrsfs,amsthm}
\usepackage{algorithm}
\usepackage{algpseudocode}

\allowdisplaybreaks[4]
\usepackage{xltabular}

\UseRawInputEncoding
\usepackage{epsfig}
\newcommand{\chuhao}{\fontsize{19pt}{\baselineskip}\selectfont}

\definecolor{purple}{rgb}{0.00,0.00,0.00}
\newcommand{\purple}[1]{\textcolor{purple}{#1}}
\newcommand{\dH}{d_{\mathrm H}}
\newcommand{\dTV}{d_{\mathrm{TV}}}
\newcommand{\dKL}{d_{\mathrm{KL}}}

\newcommand{\Ent}{\mathrm{Ent}}
\newcommand{\E}{\mathbb{E}}

\newcommand{\osc}{\mathrm{osc}}
\newcommand{\Var}{\operatorname{Var}}
\newcommand{\Cov}{\operatorname{Cov}}
\newcommand{\cG}{\mathcal{G}}
\newcommand{\cU}{\mathcal{U}}
\newcommand{\cV}{\mathcal{V}}
\newcommand{\cW}{\mathcal{W}}
\newcommand{\cE}{\mathcal{E}}

\newcommand{\logit}{\operatorname{logit}}

\numberwithin{equation}{section}

 \newtheorem{theorem}{Theorem}[section]
 \newtheorem{lemma}{Lemma}[section]
 \newtheorem{assumption}{Assumption}[section]
 \newtheorem{proposition}{Proposition}[section]
 
 \newtheorem{remark}{Remark}[section]

 \newtheorem{corollary}{Corollary}[section]

\title{\bf\color{black} \chuhao{Second-order perturbation bounds for Gibbs samplers under strong spatial mixing}}
\date{}

\begin{document}
\author{
Na Lin 
\and Aaron Smith 
\and Yiqiang Q. Zhao
}

\maketitle
\begin{abstract}

The basic question in perturbation analysis of Markov chains is how small changes in their transition kernels affect their stationary distributions. Classical perturbation bounds typically require the kernel error to be much smaller than $1/\tau$, where $\tau$ is a mixing or relaxation time. Although this scaling is sharp for Markov chains in general, we investigate a general ``square-rooting'' phenomenon in which one-step errors of order roughly $1/\sqrt{\tau}$ can be sufficient for local updates. We proved a form of this phenomenon in \cite{lin2025perturbation} under strong assumptions. Here we substantially weaken these assumptions, and prove this phenomenon occurs using three distinct approaches. First, block factorization applies under structural assumptions on the stationary measures of both chains. Second, approximate block-update arguments extend the result to statistically relevant Markov chain Monte Carlo (MCMC) settings, where structural guarantees are available for the exact posterior and its associated sampler, but not for the perturbed posterior. Third, we use direct calculations for a class of models with hard constraints where neither general result is directly available. We illustrate these results in three MCMC settings and show how they directly inform the tuning of approximate MCMC algorithms.

\vspace{0.2cm}

\end{abstract}

\section{Introduction}

In computational Bayesian statistics, it is common to write down an ``ideal'' MCMC algorithm whose stationary distribution is \textit{exactly} the posterior distribution of interest, but then implement and run an ``approximate'' MCMC algorithm whose transition rule and stationary distribution are merely close to the ideal chain. The main motivation is that ideal MCMC chains can be computationally prohibitive, and approximate chains can give similar statistical performance at a smaller computational cost. There is now a large literature on approximate MCMC chains in many settings, based on techniques such as data subsampling \cite{Korattikara2013AusterityIM,bardTall17,QuirozKohnVillaniTran2019}, data compression \cite{HugginsAdamsBroderick2017}, local pre-computation summaries \cite{BolandFrielMaire2018,LiSmith2026LPM}, and other approximate transition rules \cite{ChristenFox2005}.

Most approximate MCMC schemes come with tuning parameters that can be adjusted to increase accuracy at the cost of increased computational time. This leads to the main mathematical question in the perturbation analysis of Markov chains: if $K$ denotes the ideal transition kernel with invariant distribution $\mu$, and $\widetilde K$ denotes the implemented kernel with invariant distribution $\widetilde \mu$,
how small must the one-step error between $K$ and $\widetilde K$ be in order
to guarantee that $\mu$ and $\widetilde\mu$ are close?

Classical perturbation bounds answer this question in terms of a mixing or
relaxation scale.  Ignoring details of exactly what loss $d$ is being considered, they often show that if the perturbation error is bounded by
\[
    \sup_x d(K(x,\cdot),\widetilde K(x,\cdot)) \leq \delta
\]
and the ideal chain mixes on time scale $\tau$, then the stationary error is
controlled by a quantity of order $\tau\delta$ (see the survey \cite{rudolf2026perturbations}).
For high-temperature random-scan Gibbs samplers on \(n\) local variables, one usually has \(\tau\asymp n\), up to logarithmic factors depending on the precise notion of mixing. In this common regime, classical perturbation bounds therefore require one-step errors of order \(o(n^{-1})\).
Such estimates are known to be sharp for generic Markov chains, but they can be pessimistic for many specific chains.

To see this, we consider the case of a Gibbs sampler that targets $n$  i.i.d. Bernoulli$(\frac{1}{2})$ random variables, and a ``perturbed" Gibbs sampler that targets $n$ i.i.d. Bernoulli$(\frac{1}{2} + \delta_{n})$. Both Gibbs samplers have relaxation time $\Theta(n)$ and mixing time $\Theta(n \log(n))$ as long as $\delta_{n} = o(1)$ (see \cite{Levin2008MarkovCA}). It is also known that the total variation distance between these two stationary measures is of order $\min(1,\sqrt{n} \delta_{n})$.\footnote{This is an elementary calculation, but it is perhaps not well known and so we give a reference. See \textit{e.g.} Sections 1.3 and 1.4 of \cite{chat19} for full details. We highlight one element of the proof that seems relevant for this discussion: the total variation distance between Bernoulli \textit{vectors} is small if and only if the total variation distance between the \textit{sums} is small. In this very precise sense, we get the familiar square-root from the central limit theorem for full vectors rather than merely averages. This heuristic - that vectors are distinguishable essentially when sums are distinguishable - motivates the rest of this work.}
In particular, the perturbation is small as long as $\delta_{n} = o(n^{-\frac{1}{2}})$ is small compared to one over the square-root of the relaxation time. The reason classical perturbation bounds can be pessimistic is that worst-case perturbation bounds aggregate local errors as if they could all act ``in the same direction.'' In the independent-spin example, however, the squared local errors add rather than the errors themselves.
More generally, for graphical models at high temperature, correlations decay spatially, so local errors tend to cancel out before they can interact. This leads to the heuristic that many
small local errors should accumulate like weakly correlated errors.
This suggests that the stationary measure should have error roughly $\sqrt{\tau}\,\delta$ rather than $\tau\delta$. We refer to this as the square-rooting phenomenon.

A first result of this type was obtained by the authors in \cite{lin2025perturbation} under strong assumptions on both the form of the stationary measure and the updates. We don't see any way to directly compare the results of that paper to the results in this paper, and don't believe that either are immediate corollaries of the other. However, we believe that the current paper is typically stronger for most applications in statistics, barring the case that the dependency graph is quite tree-like. We discuss partial comparisons in Section~\ref{subsubsec:comparison-previous}, showing both (i) that Theorem~\ref{thm:main} gives a stronger estimate in the common setting of local Gibbs perturbations and (ii) that the factorization assumption in \cite{lin2025perturbation}  directly implies that its bounds must be inefficient for a common class of models.

No single argument covers all regimes in which this phenomenon occurs. We therefore develop three complementary approaches with different sets of assumptions that may be useful in different contexts.

For our first result, we take advantage of the block-factorization result of \cite{caputo2021block}. This result requires both the ideal and perturbed chains to have stationary measures with a specific dependence structure (see Equation \eqref{DefEqFactorMsr}). See Theorem~\ref{thm:main} for the main result, Algorithm~\ref{alg:noisy-gibbs} for a concrete MCMC algorithm that it applies to, and Section~\ref{subsubsec:comparison-previous} for a partial quantitative comparison with our previous work \cite{lin2025perturbation}.

The assumption that both the ideal and perturbed chains have a simple dependence structure can often be checked for the types of models from statistical physics studied in \cite{caputo2021block}. In MCMC practice, however, it can be difficult and impractical to force the stationary measure of a perturbed Markov chain to have the same structure as the stationary measure of the original chain. Fortunately, in MCMC practice one has control over the structure of the underlying MCMC sampler, and it is typically easy to force that sampler to respect the structure of the original chain's stationary measure. See Theorem~\ref{thm:approx-block-gibbs} for our main result in this regime, and Algorithm~\ref{alg:fresh-field-gibbs} for a concrete MCMC algorithm that it applies to.

For our third result, we note that there are many examples in MCMC practice for which neither structural assumption quite holds. When they fail, we can still use direct calculations. Section~\ref{subsec:Example3} treats weighted matchings, where the hard constraint that selected edges cannot share a vertex is not covered by the block-factorization estimates used above.

Despite their different assumptions and proofs, all three results bound the squared Hellinger distance by a sum of squared local perturbations:
\[
    \dH^2(\widetilde\mu,\mu)
    \leq C\sum_i \varepsilon_i^2.
\]
The meaning of $\varepsilon_i$ changes between the three arguments. For the first result, the statement takes the following concrete form. We consider finite-spin Gibbs measures on bounded-degree factor graphs. Let
\[
    \mu(\sigma) \propto \exp\left( \sum_{\phi\in\Phi}\phi(\sigma)\right)
\]
and let $\nu$ be obtained by a local perturbation
\[
    \nu(\sigma) \propto
    \exp\left( \sum_{\phi\in\Phi}\phi(\sigma)
                  + \sum_{\phi\in\Phi} g_\phi(\sigma)\right)
\]
with ``small'' errors $\delta_\phi=\frac12\osc(g_\phi)$. We show that, under appropriate assumptions,
\[
    \dH^2(\nu,\mu)
    \leq C \sum_{\phi\in\Phi}\delta_\phi^2,
\]
where the constant is independent of the size of the graph. Consequently, if the graph has $n$ local factors, perturbations of size $o(n^{-1/2})$ are sufficient for convergence in Hellinger distance. The second and third results are inspired by this argument, though the details are more difficult.

While our paper was motivated in part by early approximate-MCMC work for tall data, including austerity MCMC \cite{Korattikara2013AusterityIM} and noisy MCMC \cite{AlquierFrielEverittBoland2016}, the perturbation bounds are general and can be used in many other settings where perturbed Markov chains appear.

\subsection{Related Work}

Perturbation theory has a long history in the stability analysis of matrices and linear operators; see, for example, the classical monograph \cite{Kato1995}. 
The first generic perturbation results specialized to Markov
chains were developed for finite or discrete-state chains in
\cite{Schweitzer1968,Kartashov1986}. Subsequent work established perturbation bounds under uniform or geometric ergodicity and in several probability metrics, including total variation, Wasserstein, and \(L^2\) distances \cite{Mitrophanov2005,AlquierFrielEverittBoland2016,rudolf18,
negrea_rosenthal_2021}. 
Despite their different assumptions and formulations, classical perturbation bounds generally control the stationary error by the product of a worst-case one-step perturbation and a mixing or relaxation scale. 
Other work has exploited specific model or algorithmic structure to obtain more informative approximation guarantees. For example, \cite{RendellJohansenLeeWhiteley2021} develops a proxy-based parallel MCMC method with a tunable trade-off between computational efficiency and fidelity to the target distribution, while our earlier work \cite{lin2025perturbation} obtains improved perturbation bounds for Gibbs samplers associated with graphical models.
We refer to \cite{rudolf2026perturbations} for a more detailed historical and technical review.

One of the major uses of perturbation estimates is in the analysis of
approximate MCMC algorithms, where an ideal MCMC algorithm is replaced by a computationally cheaper approximation. This includes subsampling \cite{Korattikara2013AusterityIM,AlquierFrielEverittBoland2016,QuirozKohnVillaniTran2019} and pre-computation methods for
intractable likelihoods \cite{BolandFrielMaire2018}. Control variates are often used to improve the accuracy of such approximations \cite{QuirozKohnVillaniTran2019,CVSGMCMC19}. 
In the present paper, these algorithms serve primarily as motivating examples; the perturbation bounds themselves apply more broadly.


\subsection{Structure of Paper}

Section~\ref{Sec-Preliminaries} fixes notation, Section~\ref{sec:MainResult} proves the two general bounds, and Section~\ref{SecWorkedExamples} gives the three collections of worked examples.

Section~\ref{SecWorkedExamples} is the largest part of the paper, so we give a short guide to the main examples. Section~\ref{subsec:Example1} is a non-statistical illustrative example. It treats the high-temperature Ising model on a square box and compares the present bounds with \cite{lin2025perturbation}. Section~\ref{subsec:Example2} studies a simple image-restoration model with Ising priors. It applies our two general results to an augmented-state sampler (which can be analyzed using techniques from \cite{caputo2021block}) and a simpler noisy sampler (which cannot). It then studies the sufficient subsample-size scaling for a simple control-variate scheme. Section~\ref{subsec:Example3} starts with generic Markov chains for sampling weighted matchings, explains the connection to Bayesian record linkage, and gives a similar applied analysis.

We also briefly comment on the organization of the main subsections of Section~\ref{SecWorkedExamples}. Our primary motivation is application to MCMC, and so it is important that we show that our generic results do apply to realistic MCMC algorithms. On the other hand, our basic arguments are quite broad, and so we are motivated to make them easily-accessible in other contexts. To accommodate both of these motivations, we treat our three approaches in roughly the same way. First, we state and prove a result for a generic class of Markov chains, without any reference to statistics or data. Next, we introduce a class of statistical models and MCMC algorithms,\footnote{Except for Section ~\ref{subsec:Example1}, which is a toy example that has no statistical content.} then prove an estimate based on somewhat abstract bounds on the quantities appearing in the algorithms. Finally, we relate these abstract bounds to concrete choices (such as the choice of control variate) and explain how our generic perturbation estimates relate to the computational costs of different parts of these more concrete algorithms.

\section{Preliminaries}\label{Sec-Preliminaries}

\subsection{Generic probability notation}

Let $\mu, \nu$ be two probability distributions on a finite set $\Omega$.
The Hellinger distance between $\mu$ and $\nu$ is defined by
\[
\dH(\nu,\mu)=\left(\sum_{\sigma \in \Omega}\left(\sqrt{\nu(\sigma)}-\sqrt{\mu(\sigma)}\right)^2 \right)^{1/2},
\]
and the total variation (TV) distance is defined by
\[
\dTV(\nu,\mu)= \frac{1}{2} \sum_{\sigma \in \Omega}|\nu(\sigma)-\mu(\sigma)|.
\]
If $\nu$ is absolutely continuous with respect to $\mu$, then the Kullback--Leibler (KL) divergence of $\nu$ relative to $\mu$ is defined as
\[
\dKL(\nu,\mu)=\sum_{\sigma \in \Omega} \nu(\sigma) \log\left(\frac{\nu(\sigma)}{\mu(\sigma)}\right).
\]
We recall the standard inequalities (see, e.g., p.~135 of \cite{daskalakis20} and Lemma 2.4 of \cite{Tsybakov2009Nonparametric}):
\begin{equation}\label{eq:H-vs-TV}
\frac{1}{2}\dH^2(\nu,\mu)\leq \dTV(\nu,\mu) \leq \dH(\nu,\mu)
\end{equation}
and
\begin{equation}
\dH^2(\nu,\mu)\ \le\ \dKL(\nu,\mu).
\label{eq:H-vs-KL}
\end{equation}

By a small abuse of notation, when $d$ is a distance on measures, we extend it to a distance on transition kernels via the formula:

\[
d(Q,K) = \sup_{x \in \Omega} d(Q(x,\cdot), K(x,\cdot)).
\]

For any function $g \, : \, \Omega \to \mathbb{R}$, we define the oscillation
\[
\osc(g) = \sup_{\sigma \in \Omega} g(\sigma) - \inf_{\sigma \in \Omega} g(\sigma).
\]

\subsection{Notation for Graphical Models}

Fix $s \in \mathbb{N}$ and define $[s] = \{1,2,\ldots,s\}$. We denote by the triple $G=(V,\Phi,E)$ a factor graph, following the same conventions as \cite{lin2025perturbation} and the work referenced therein:

\begin{itemize}
    \item $V$ is any finite set. We call this set the \textit{vertices} of the graph.
    \item $E \subseteq \binom{V}{2}$ is a set of unordered pairs of distinct vertices. We call this set the \textit{edges} of the graph.
    \item $\Phi$ is any finite indexed collection of functions from $\Omega \equiv [s]^{V}$ to $\mathbb{R}$. Each $\phi\in\Phi$ is equipped with a nonempty \textit{scope} $S[\phi]\subseteq V$ and is assumed to depend only on the coordinates in $S[\phi]$.\footnote{Informally, we think of the scope as the collection of variables appearing in the formula for $\phi$. We use this slightly awkward definition to deal with simple parametric families like $\phi(\omega) = \alpha \omega_{1},$ where we wish to think of $\phi$ as depending on $\omega_{1}$ for all parameter values even though $\phi$ does not depend on $\omega_{1}$ when $\alpha=0$.} We call this collection the \textit{factors} of the graph.
\end{itemize}

\purple{Let \(d\) denote the graph distance induced by \((V,E)\).} We abuse notation slightly by defining the distance between two vertex sets $A$ and $B$ as
\[
d(A,B)=\min\{d(x,y):x\in A,y\in B\}.
\]
For a set $A\subset V$, $A^c=V\setminus A$ and the exterior boundary is $\partial A=\{x\in A^c: d(x,A)=1\}$.
For an integer $r>0$, we denote by $B_r(A)$ the set of vertices whose distance from $A$ is strictly less than $r$:
\[
B_r(A)=\{x\in V: d(x,A)< r\}.
\]
For $x\in V$, we write $B_r(x)=B_r(\{x\})$.

We define the Gibbs measure associated with a factor graph to be the probability measure on $\Omega$ given by:
\begin{equation}\label{DefEqFactorMsr}
\mu(\sigma)\propto\exp\left(\sum_{\phi\in \Phi}\phi(\sigma)\right),
\end{equation}
where $\sigma$ represents a configuration in $\Omega$.
The declared scopes let us view the factor graph as a bipartite graph in the usual sense, with vertices corresponding to $V\cup\Phi$ and an edge between $v$ and $\phi$ whenever $v\in S[\phi]$. With this in mind, for $v\in V$, denote by
\[
E[v]=\{\phi\in\Phi:v\in S[\phi]\}
\]
the set of factors whose declared scope contains $v$. Throughout this paper, we denote by
\[
\Delta=\max_{x\in V}|E[x]|<\infty,
\qquad
\kappa=\max_{\phi\in\Phi}|S[\phi]|<\infty,
\]
the maximum degrees of the two types of vertices in this bipartite graph.

For vertices $A \subset V$ and configuration $\sigma \in \Omega$, we denote by $\sigma_A \in [s]^A$ the restriction of the configuration $\sigma$ to $A$, i.e.,
\[
\sigma_A(x) = \sigma(x), \quad x \in A.
\]
For a measure $\mu$ on $\Omega$, we let $\mu_A$ denote the marginal of $\mu$ on $[s]^{A}$, defined by
\[
\mu_A(\sigma_A) = \sum_{\eta \in \Omega:\, \eta_A = \sigma_A} \mu(\eta),
\quad \sigma_A \in [s]^A.
\]
For $A \subset B \subset V$ and $\tau \in \Omega$ \purple{with \(\mu_{B^c}(\tau_{B^c})>0\)}, let $\mu_A^{\tau_{B^c}}$ denote the marginal on $A$ of the conditional distribution of $\mu$ given $\eta_{B^c} = \tau_{B^c}$, that is,
\[
\mu_A^{\tau_{B^c}}(\sigma_A)
=
\frac{\sum_{\eta \in \Omega:\, \eta_A = \sigma_A,\ \eta_{B^c} = \tau_{B^c}} \mu(\eta)}
{\sum_{\eta \in \Omega:\, \eta_{B^c} = \tau_{B^c}} \mu(\eta)}=\frac{\sum_{\eta \in \Omega:\, \eta_A = \sigma_A,\ \eta_{B^c} = \tau_{B^c}} \mu(\eta)}
{\mu_{B^c}(\tau_{B^c})}.
\]
\purple{Measures of the form \eqref{DefEqFactorMsr} have full support, so these conditionals are defined for every boundary configuration.}
For $\sigma_A \in [s]^A$ and $\tau_{A^c} \in [s]^{A^c}$, we denote by $\sigma_A \tau_{A^c}$ the unique element of $\Omega$ whose restriction to $A$ is $\sigma_{A}$ and whose restriction to $A^{c}$ is $\tau_{A^{c}}$. We then have the following conditional probability formula, in this notation:
\begin{equation}\label{ConditionlProbability}
\mu_A^{\tau_{A^c}}(\sigma_A)
=\frac{\mu(\sigma_A \tau_{A^c})}{\mu_{A^c}(\tau_{A^c})}.
\end{equation}

For any measurable function $f : \Omega \to \mathbb{R}_+$, let
\[
\mu_A^{\tau_{A^c}}(f)
=
\sum_{\sigma_A \in [s]^A} \mu_A^{\tau_{A^c}}(\sigma_A)\, f(\sigma_A \tau_{A^c})
\]
denote the expectation of $f$ under $\mu_A^{\tau_{A^c}}$. We define the function $\mu_A f : [s]^{A^c} \to \mathbb{R}$ by
\[
(\mu_A f)(\tau_{A^c}) = \mu_A^{\tau_{A^c}}(f).
\]

For $A \subset V$ and $\tau \in \Omega$, the local entropy of a nonnegative function $f$ with respect to $\mu_A^{\tau_{A^c}}$ is defined by
\begin{equation}\label{Def:LocalEntropy}
\operatorname{Ent}_A^{\tau_{A^c}}(f)
=
\mu_A^{\tau_{A^c}} \!\left[
f \log \!\left( \frac{f}{\mu_A^{\tau_{A^c}}(f)} \right)
\right].
\end{equation}
We further define the function $\operatorname{Ent}_A f : [s]^{A^c} \to \mathbb{R}_+$ by
\[
(\operatorname{Ent}_A f)(\tau_{A^c}) = \operatorname{Ent}_A^{\tau_{A^c}}(f).
\]
In the special case where $A = V$, so that $A^c = \emptyset$ and $\mu_A^{\tau_{A^c}} = \mu$, we just define
\[
\operatorname{Ent}_V(f)
=
\mu \!\left[
f \log \!\left( \frac{f}{\mu(f)} \right)
\right].
\]

We recall the following equalities. The first follows directly from the definition of the KL divergence, while the second is an immediate consequence of the conditional entropy decomposition (see, e.g., (2.10) in \cite{caputo2015approximate}) together with the chain rule for KL divergence (see Theorem 2.5.3 of \cite{cover_thomas_2006}):

\begin{lemma}
\label{lem:Ent-is-KL}
Suppose that $\nu$ is absolutely continuous with respect to $\mu$, and let $f=\frac{\nu}{\mu}$.
Then
\[
\Ent_V(f)=\dKL(\nu,\mu).
\]
Moreover, for any $A\subset V$,
\begin{equation}
\mu(\Ent_A f)
=
\E_{\nu_{A^c}}\!\left[
\dKL\!\left(\nu_A^{\tau_{A^c}},\mu_A^{\tau_{A^c}}\right)
\right].
\label{eq:blockKL}
\end{equation}
\end{lemma}

\subsection{Perturbed Gibbs measure}

Let $\mu$ be the Gibbs measure defined in \eqref{DefEqFactorMsr}. For a measurable function $g: \Omega\to \mathbb{R}$, we define the perturbed Gibbs measure $\nu$ by
\begin{equation}\label{Perturbed-Measure}
\nu(\sigma)\propto\exp\left(\sum_{\phi\in \Phi}\phi(\sigma)+g(\sigma)\right).
\end{equation}
That is, $\nu$ is obtained from $\mu$ by an exponential tilt:
\begin{equation}\label{TiltedForm}
\nu(\sigma)=
\frac{e^{g(\sigma)}\mu(\sigma)}{\mu(e^g)},
\qquad \sigma\in\Omega.
\end{equation}
In particular, $\nu$ is absolutely continuous with respect to $\mu$, i.e., $\nu \ll \mu$.

This formulation provides a unified framework for incorporating a broad class of perturbations. The function $g$ may represent modifications of external fields, interaction potentials, or more general energy contributions. In particular, if $g$ admits a decomposition of the form
\begin{equation}\label{DecompositionOfG}
g(\sigma)=\sum_{\phi \in \Phi} g_{\phi}(\sigma),
\end{equation}
where each $g_{\phi}$ depends only on the coordinates in $S[\phi]$, then $\nu$ remains a Gibbs measure with modified potentials $\phi + g_{\phi}$. We define the size of the local perturbation by
\[
\delta_{\phi}
=
\frac12\osc(g_{\phi})
=
\inf_{c\in\mathbb R}\|g_{\phi}-c\|_\infty,
\]
where $\|\cdot\|_\infty$ denotes the $L^\infty$ norm. This definition is unchanged if a constant is added to $g_{\phi}$, as it should be because such a change does not alter the Gibbs measure.

{\color{purple}Section~\ref{sec:MainResult} uses this setup first to compare two stationary measures and then to compare exact and approximate local updates.}

\section{Main Results}\label{sec:MainResult}

{\color{purple}We prove two general results. Section~\ref{subsec:main-setup} gives the factorization assumptions used below. Section~\ref{subsec:local-gibbs-perturbations} assumes that the perturbed stationary measure has the same local Gibbs structure as the ideal measure, but makes no assumption on the update rule; this result only requires approximate tensorization over single sites. Section~\ref{subsec:approx-block-gibbs} assumes that the perturbed kernel uses random-scan block updates, but makes no assumption on the form of its invariant measure; this result only requires block factorization for the scan used by the kernel. The stronger Assumption~\ref{assume-block-factor}, which is available in all of our block-factorization applications, implies both hypotheses.}

\subsection{Setup}\label{subsec:main-setup}

One of our key assumptions is the following \textit{strong spatial mixing} (SSM) condition, which quantifies how the influence of a perturbation at a boundary site on a finite set $A$ decays exponentially with the distance to $A$. There are many equivalent notions of strong spatial mixing; we use the definition adopted in \cite{Cesi2001} and used by \cite{caputo2021block}:

\begin{assumption}\label{assume-SSM}
There exist constants \purple{$0<M,m<\infty$} such that for any sets $A \subset B \subset V$, any site $x \in \partial B$, and any boundary configurations $\tau_{B^c}, \tau'_{B^c} \in [s]^{B^c}$ satisfying $\tau(y) = \tau'(y)$ for all $y \in B^c\setminus\{x\}$, we have
\[
\left\| \frac{\mu_{A}^{\tau_{B^c}}}{\mu_{A}^{\tau'_{B^c}}} - 1 \right\|_\infty \leq M e^{-m\, d(x,A)}.
\]
\end{assumption}

\begin{remark}[Satisfying the SSM condition]
There are many notions of strong spatial mixing. Famously (though non-obviously), there is a sense in which they are equivalent (see, e.g., Theorem 2.1 of \cite{DobrushinShlosman1987CAConstructive}). We note that Assumption~\ref{assume-SSM} implies the classical strong spatial mixing condition in TV distance with the same exponential rate and prefactor $M/2$, since for any $\tau_{B^c}, \tau'_{B^c} \in [s]^{B^c}$ satisfying $\tau(y) = \tau'(y)$ for all $y \in B^c\setminus\{x\}$,
\begin{align*}
\dTV(\mu_{A}^{\tau_{B^c}},\mu_{A}^{\tau'_{B^c}}) &=\frac{1}{2}\sum_{\sigma\in [s]^A} |\mu_{A}^{\tau_{B^c}}(\sigma)-\mu_{A}^{\tau'_{B^c}}(\sigma)|\\
&= \frac{1}{2}\sum_{\sigma\in [s]^A} \mu_{A}^{\tau'_{B^c}}(\sigma)\left|\frac{\mu_{A}^{\tau_{B^c}}(\sigma)}{\mu_{A}^{\tau'_{B^c}}(\sigma)} -1\right|\\
& \leq \frac{1}{2}\sum_{\sigma\in [s]^A} \mu_{A}^{\tau'_{B^c}}(\sigma) \cdot M e^{-m\, d(x,A)} =\frac{M}{2} e^{-m\, d(x,A)}.
\end{align*}
\purple{For the Ising models considered in Sections~\ref{subsec:Example1} and~\ref{subsec:Example2}, Lemma~\ref{lem:ising-dobrushin-block-factorization} verifies Assumption~\ref{assume-SSM} directly from the Dobrushin bound used there.}

\end{remark}

We use the following block-factorization assumption.

\begin{assumption}[Block factorization]\label{assume-block-factor}
There exists a positive constant $C_{\mathrm{BF}}<\infty$ such that, for any family of blocks $\mathcal{A} \subseteq 2^V$, any choice of nonnegative weights $\alpha = \{\alpha_A\}_{A \in \mathcal{A}}$, and any function $f : \Omega \to \mathbb{R}_+$ with $f \log^+ f \in L^1(\mu)$,
\begin{equation}\label{Ineq-BlockFactor}
\gamma(\alpha)\,\mathrm{Ent}_V(f)
\leq C_{\mathrm{BF}} \sum_{A \in \mathcal{A}} \alpha_A \,\mu\big(\mathrm{Ent}_A f\big),
\end{equation}
where
\[
\gamma(\alpha) = \min_{x \in V} \sum_{A\in \mathcal{A}:\, x\in A} \alpha_A.
\]
\end{assumption}

\begin{remark}
For all of the examples studied with block-factorization techniques in this paper, Assumption~\ref{assume-SSM} implies Assumption~\ref{assume-block-factor} by \purple{\cite[Theorem~2.3 and Section~2.2.4]{caputo2021block}}. Thus Assumption~\ref{assume-block-factor} does not need to be checked separately in any application in this paper.

We state Assumption~\ref{assume-block-factor} for general factor graphs in order to ``future-proof" our result. There is quite a lot of recent work in extending results such as \cite[Theorem~2.3]{caputo2021block} to more general factor graphs, and we expect to see statistically-relevant extensions in the future. See, for example, \cite{Chen2024Factorization,CaputoChenParisi2026}.
\end{remark}

The first result below only needs the singleton consequence of Assumption~\ref{assume-block-factor}.

\begin{assumption}[Approximate tensorization]\label{assume-approx-tensor}
There exists a constant $C_{\mathrm{AT}}<\infty$ such that, for every function $f:\Omega\to\mathbb R_+$ with $f\log^+f\in L^1(\mu)$,
\begin{equation}\label{Ineq-ApproxTensor}
\Ent_V(f)
\le
C_{\mathrm{AT}}
\sum_{x\in V}\mu\bigl(\Ent_{\{x\}}f\bigr).
\end{equation}
\end{assumption}

Assumption~\ref{assume-block-factor} implies Assumption~\ref{assume-approx-tensor} with $C_{\mathrm{AT}}=C_{\mathrm{BF}}$ by taking the singleton blocks $\mathcal A=\{\{x\}:x\in V\}$ with unit weights.

Combining Inequality~\eqref{eq:H-vs-KL} with Lemma~\ref{lem:Ent-is-KL} and Assumption~\ref{assume-block-factor}, we immediately obtain the following corollary.

\begin{corollary}\label{Corollary-H2-BlockKL}
Suppose that Assumption~\ref{assume-block-factor} holds and that $\nu$ is absolutely continuous with respect to $\mu$. Then, with $C_{\mathrm{BF}}$ as in Assumption~\ref{assume-block-factor}, for any family of blocks $\mathcal{A} \subseteq 2^V$ and any choice of nonnegative weights $\alpha = \{\alpha_A\}_{A \in \mathcal{A}}$ with $\gamma(\alpha)>0$,
\begin{equation}\label{Ineq-KL-Factor}
\dH^2(\nu,\mu)\leq \dKL(\nu,\mu)
\leq \frac{C_{\mathrm{BF}}}{\gamma(\alpha)} \sum_{A \in \mathcal{A}} \alpha_A \, \E_{\nu_{A^c}} \left[\dKL\left(\nu_A^{\tau_{A^c}},\mu_A^{\tau_{A^c}}\right)\right].
\end{equation}
\end{corollary}

\subsection{\purple{Perturbing the Stationary Measure}}\label{subsec:local-gibbs-perturbations}

We will also use the following estimate on the KL divergence between two measures in terms of the oscillation of the tilt:

\begin{lemma}
\label{lem:tilt}
Let $\mu$, $\nu$ be Gibbs measures on a finite set $\Omega$ defined as in \eqref{DefEqFactorMsr} and \eqref{Perturbed-Measure} with perturbation $g:\Omega \to \mathbb{R}$.
Then
\[
\dKL(\nu, \mu) \le \frac{\osc^2(g)}{8}.
\]
\end{lemma}

\begin{proof}
Set
\[
\psi(t)=\log\mu(e^{tg}),
\qquad
\mu_t(\sigma)=\frac{e^{tg(\sigma)}\mu(\sigma)}{\mu(e^{tg})},
\qquad 0\le t\le 1.
\]
Then \(\psi'(t)=\mu_t(g)\) and \(\psi''(t)=\Var_{\mu_t}(g)\). By \eqref{TiltedForm},
\[
\dKL(\nu,\mu)=\psi'(1)-\psi(1).
\]
Since \(\psi(0)=0\), integration by parts gives
\[
\dKL(\nu,\mu)
=
\int_0^1 t\,\Var_{\mu_t}(g)\,dt.
\]
Since a random variable taking values in an interval of length \(\osc(g)\) has variance at most \(\osc^2(g)/4\),
\[
\dKL(\nu,\mu)
\le
\frac{\osc^2(g)}{4}\int_0^1 t\,dt
=
\frac{\osc^2(g)}{8}.
\]
\end{proof}

Lemma~\ref{lem:tilt} provides a global bound on the KL divergence between $\mu$ and $\nu$. We now turn to a localized version and derive an upper bound on the conditional KL divergence for Gibbs measures restricted to a block.
Recall that for $x\in V$, we let
\[
E[x]=\{\phi \in \Phi:x\in S[\phi]\}.
\]
For $A\subseteq V$, we define the set of factors whose declared scope intersects $A$ by
\[
E[A]=\bigcup_{x\in A} E[x].
\]

\begin{lemma}
\label{lem:UpperBoundlocalKL}
Let $\mu$, $\nu$ be Gibbs measures on a finite set $\Omega$, defined as in \eqref{DefEqFactorMsr} and \eqref{Perturbed-Measure}, with perturbation $g:\Omega \to \mathbb{R}$ admitting the decomposition \eqref{DecompositionOfG}. Then for any $A\subseteq V$ and $\tau_{A^c}\in [s]^{A^c}$,
\begin{equation}\label{eq:localKL}
\dKL(\nu_A^{\tau_{A^c}},\mu_A^{\tau_{A^c}})\ \le\ \frac12\Big(\sum_{\phi\in E[A]}\delta_\phi\Big)^2
\ \le\ \frac{|E[A]|}{2}\sum_{\phi\in E[A]}\delta_\phi^2.
\end{equation}
\end{lemma}

\begin{proof}
Fix $A \subseteq V$ and a boundary condition $\tau_{A^c}$. Both $\mu_A^{\tau_{A^c}}$ and $\nu_A^{\tau_{A^c}}$ are Gibbs measures on $[s]^A$. All terms that can depend on $\sigma_A$ are indexed by $\phi \in E[A]$; terms outside $E[A]$ are constant on the conditional fiber.

For any $\eta_A \in [s]^A$, let $g_{A}(\eta_A) = \sum_{\phi \in E[A]} g_\phi(\eta_A \, \tau_{A^c})$ collect the perturbation terms restricted to $A$.
It follows that
\[
\frac{\nu_A^{\tau_{A^c}}(\eta_A)}{\mu_A^{\tau_{A^c}}(\eta_A)}
= \frac{e^{g_{A}(\eta_A)}}{\mu_A^{\tau_{A^c}}(e^{g_{A}})}.
\]

By definition, $\osc(g_\phi)=2\delta_\phi$. Restricting $g_\phi$ to the conditional fiber can only decrease its oscillation. Hence,
\begin{align*}
\osc(g_{A})&= \sup_{\eta_A \in [s]^A} g_A(\eta_A)- \inf_{\eta_A \in [s]^A} g_A(\eta_A)\\
&=\sup_{\eta_A \in [s]^A} \sum_{\phi \in E[A]} g_\phi(\eta_A \, \tau_{A^c})- \inf_{\eta_A \in [s]^A} \sum_{\phi \in E[A]} g_\phi(\eta_A \, \tau_{A^c}) \\
&\leq \sum_{\phi \in E[A]} \left|\sup_{\eta_A \in [s]^A} g_\phi(\eta_A \, \tau_{A^c}) -\inf_{\eta_A \in [s]^A} g_\phi(\eta_A \, \tau_{A^c})\right| \le 2 \sum_{\phi \in E[A]} \delta_\phi.
\end{align*}
Applying Lemma~\ref{lem:tilt} yields
\[
\dKL\big(\nu_A^{\tau_{A^c}}, \mu_A^{\tau_{A^c}}\big)
\le \frac{1}{2} \Big(\sum_{\phi \in E[A]} \delta_\phi \Big)^2,
\]
which proves the first inequality in \eqref{eq:localKL}.

The second inequality follows from the Cauchy--Schwarz inequality:
\[
\Big(\sum_{\phi \in E[A]} \delta_\phi \Big)^2
\le |E[A]| \sum_{\phi \in E[A]} \delta_\phi^2.
\]
\end{proof}

The preceding local estimate and approximate tensorization yield the following main result.

\begin{theorem}\label{thm:main}
Let $\mu$, $\nu$ be Gibbs measures on a finite set $\Omega$, defined as in \eqref{DefEqFactorMsr} and \eqref{Perturbed-Measure}, with perturbation $g:\Omega\to\mathbb R$ admitting the decomposition \eqref{DecompositionOfG}. Suppose that Assumption~\ref{assume-approx-tensor} holds. Then
\begin{equation}\label{eq:main}
\dH^2(\nu,\mu)
\le
\dKL(\nu,\mu)
\le
\frac{C_{\mathrm{AT}}\Delta\kappa}{2}
\sum_{\phi\in\Phi}\delta_\phi^2,
\end{equation}
where $C_{\mathrm{AT}}$ is the constant in Assumption~\ref{assume-approx-tensor}.
\end{theorem}

\begin{proof}
Let $f=\nu/\mu$. By Lemma~\ref{lem:Ent-is-KL} and Assumption~\ref{assume-approx-tensor},
\begin{align*}
\dKL(\nu,\mu)
&=
\Ent_V(f)\\
&\le
C_{\mathrm{AT}}
\sum_{x\in V}\mu\bigl(\Ent_{\{x\}}f\bigr)\\
&=
C_{\mathrm{AT}}
\sum_{x\in V}
\E_{\nu_{\{x\}^c}}
\left[
\dKL\left(
\nu_{\{x\}}^{\tau_{\{x\}^c}},
\mu_{\{x\}}^{\tau_{\{x\}^c}}
\right)
\right].
\end{align*}
Applying Lemma~\ref{lem:UpperBoundlocalKL} with $A=\{x\}$ and using $|E[x]|\le\Delta$ gives
\[
\dKL(\nu,\mu)
\le
\frac{C_{\mathrm{AT}}\Delta}{2}
\sum_{x\in V}\sum_{\phi\in E[x]}\delta_\phi^2.
\]
Each factor $\phi$ appears once for every $x\in S[\phi]$, and $|S[\phi]|\le\kappa$. Therefore
\[
\sum_{x\in V}\sum_{\phi\in E[x]}\delta_\phi^2
=
\sum_{\phi\in\Phi}|S[\phi]|\delta_\phi^2
\le
\kappa\sum_{\phi\in\Phi}\delta_\phi^2.
\]
This proves the KL bound in \eqref{eq:main}; the Hellinger bound follows from \eqref{eq:H-vs-KL}.
\end{proof}

\begin{corollary}\label{corollary:main}
For each $n\in\mathbb N$, let $G_n=(V_n=[n],\Phi_n,E_n)$ be a factor graph whose maximum degrees are bounded uniformly by $\Delta<\infty$ and $\kappa<\infty$. Let $\mu^{(n)}$, $\nu^{(n)}$ be the Gibbs measures on $\Omega_n=[s]^n$, defined as in \eqref{DefEqFactorMsr} and \eqref{Perturbed-Measure}, with perturbations $g^{(n)}$ admitting decompositions of the form \eqref{DecompositionOfG}. For $\phi\in\Phi_n$, write
\[
\delta_\phi^{(n)}
=
\frac12\osc\bigl(g_\phi^{(n)}\bigr).
\]
Suppose that Assumption~\ref{assume-approx-tensor} holds for $\mu^{(n)}$ with a constant $C_{\mathrm{AT}}$ independent of $n$. If
\begin{equation}\label{Eq:PerturbationL2Bound}
\sum_{\phi\in\Phi_n}\bigl(\delta_\phi^{(n)}\bigr)^2
\longrightarrow 0,
\end{equation}
then
\[
\dH\bigl(\nu^{(n)},\mu^{(n)}\bigr)
\longrightarrow 0.
\]
In particular, \eqref{Eq:PerturbationL2Bound} holds if
\[
\max_{\phi\in\Phi_n}\delta_\phi^{(n)}
=
o\bigl(n^{-1/2}\bigr).
\]
\end{corollary}

\begin{proof}
The first assertion follows directly from Theorem~\ref{thm:main}. For the final assertion, every declared scope is nonempty, so every factor in $\Phi_n$ belongs to $E[x]$ for at least one $x\in V_n$. Hence $|\Phi_n|\le n\Delta$, and
\[
\sum_{\phi\in\Phi_n}\bigl(\delta_\phi^{(n)}\bigr)^2
\le
n\Delta
\left(
\max_{\phi\in\Phi_n}\delta_\phi^{(n)}
\right)^2
=
o(1).
\]
\end{proof}

\subsection{\purple{Perturbing the Updates}}\label{subsec:approx-block-gibbs}

We now consider approximate block Gibbs updates directly.
\purple{In contrast to Theorem~\ref{thm:main}, the result in this subsection does not assume that the perturbed invariant distribution is itself a Gibbs measure of the form \eqref{Perturbed-Measure}. The ideal measure \(\mu\) need only satisfy the scan-specific factorization assumption below, while the perturbed kernel must be a random-scan block kernel whose updates approximate the corresponding exact conditionals.}
Let \(\mathcal A\subseteq 2^V\) be a family of blocks, and let
\(p=\{p_A\}_{A\in\mathcal A}\) be nonnegative scan probabilities satisfying
\[
    \sum_{A\in\mathcal A}p_A=1,
    \qquad
    \gamma(p):=
    \min_{x\in V}
    \sum_{A\in\mathcal A:\,x\in A}p_A
    >0 .
\]

\begin{assumption}[Scan-specific block factorization]\label{assume-scan-specific-factor}
There exists a constant $C_p<\infty$ such that, for every function $f:\Omega\to\mathbb R_+$ with $f\log^+f\in L^1(\mu)$,
\begin{equation}\label{Ineq-ScanSpecificFactor}
\gamma(p)\Ent_V(f)
\le
C_p
\sum_{A\in\mathcal A}p_A\,\mu\bigl(\Ent_Af\bigr).
\end{equation}
\end{assumption}
Assumption~\ref{assume-block-factor} implies Assumption~\ref{assume-scan-specific-factor} with $C_p=C_{\mathrm{BF}}$, but Theorem~\ref{thm:approx-block-gibbs} only needs \eqref{Ineq-ScanSpecificFactor} for the particular scan $p$.

At each step, the sampler chooses block \(A\in\mathcal A\) with probability \(p_A\).  Given the current configuration \(\tau\), the exact block Gibbs update would keep \(\tau_{A^c}\) fixed and resample \(\tau_A\in [s]^A\) from \(\mu_A^{\tau_{A^c}}\). We instead allow an approximate conditional distribution \(q_A^{\tau_{A^c}}\) on \([s]^A\).  Define
\begin{equation*}\label{eq:approx-block-component}
    \widetilde K_A(\tau,\sigma)
    =
    \mathbf 1\{\sigma_{A^c}=\tau_{A^c}\}
    q_A^{\tau_{A^c}}(\sigma_A),
    \qquad \tau,\sigma \in\Omega,
\end{equation*}
and set
\begin{equation*}\label{eq:approx-block-kernel}
    \widetilde K
    =
    \sum_{A\in\mathcal A}p_A\widetilde K_A.
\end{equation*}
Let \(\widetilde\mu\) be any invariant distribution of \(\widetilde K\).

\begin{theorem}\label{thm:approx-block-gibbs}
Suppose that Assumption~\ref{assume-scan-specific-factor} holds. Then
\begin{equation}\label{eq:approx-block-gibbs}
\dH^2(\widetilde\mu,\mu)
\le
\dKL(\widetilde\mu,\mu)
\le
\frac{C_p}{\gamma(p)}
\sum_{A\in\mathcal A}p_A
\E_{\widetilde\mu_{A^c}}
\left[
\dKL\left(q_A^{\tau_{A^c}},\mu_A^{\tau_{A^c}}\right)
\right],
\end{equation}
where \(C_p\) is the constant in Assumption~\ref{assume-scan-specific-factor}.
\end{theorem}

\begin{proof}
Since \(\mu\) has full support, \(\widetilde\mu\ll\mu\). By invariance of
\(\widetilde\mu\) and convexity of KL divergence in its first argument,
\begin{equation}\label{eq:approx-block-convexity}
\dKL(\widetilde\mu,\mu)
=
\dKL\left(
\sum_{A\in\mathcal A}p_A\widetilde\mu\widetilde K_A,
\mu
\right)
\le
\sum_{A\in\mathcal A}p_A
\dKL(\widetilde\mu\widetilde K_A,\mu).
\end{equation}
For a fixed block \(A\), the measure \(\widetilde\mu\widetilde K_A\) has
\(A^c\)-marginal \(\widetilde\mu_{A^c}\) and conditional distribution
\(q_A^{\tau_{A^c}}\) on \(A\). Indeed,
\begin{align*}
\widetilde{\mu} \widetilde{K}_A(\sigma) &= \sum_{\tau} \widetilde{\mu}(\tau) \widetilde{K}_A(\tau, \sigma) \\
&= \sum_{\tau_A, \tau_{A^c}} \widetilde{\mu}(\tau_A\tau_{A^c})\mathbf 1\{\sigma_{A^c}=\tau_{A^c}\} q_A^{\tau_{A^c}}(\sigma_A)  \\
&= \sum_{\tau_A} \widetilde{\mu}(\tau_A\sigma_{A^c}) q_A^{\tau_{A^c}}(\sigma_A) = \widetilde{\mu}_{A^c}(\sigma_{A^c})q_A^{\tau_{A^c}}(\sigma_A).
\end{align*}
The chain rule for KL divergence therefore gives
\begin{equation}\label{eq:approx-block-after-update}
\dKL(\widetilde\mu\widetilde K_A,\mu)
=
\dKL(\widetilde\mu_{A^c},\mu_{A^c})
+
\E_{\widetilde\mu_{A^c}}
\left[
\dKL\left(q_A^{\tau_{A^c}},\mu_A^{\tau_{A^c}}\right)
\right].
\end{equation}
The same chain rule applied to \(\widetilde\mu\) gives
\begin{equation}\label{eq:approx-block-before-update}
\dKL(\widetilde\mu,\mu)
=
\dKL(\widetilde\mu_{A^c},\mu_{A^c})
+
\E_{\widetilde\mu_{A^c}}
\left[
\dKL\left(\widetilde\mu_A^{\tau_{A^c}},\mu_A^{\tau_{A^c}}\right)
\right].
\end{equation}
Substituting \eqref{eq:approx-block-after-update} and
\eqref{eq:approx-block-before-update}, multiplied by \(\sum_A p_A=1\), into \eqref{eq:approx-block-convexity}, and canceling the common marginal terms yields
\begin{align*}
\sum_{A\in\mathcal A}p_A
\E_{\widetilde\mu_{A^c}}
\left[
\dKL\left(\widetilde\mu_A^{\tau_{A^c}},\mu_A^{\tau_{A^c}}\right)
\right]\le
\sum_{A\in\mathcal A} p_A
\E_{\widetilde\mu_{A^c}}
\left[
\dKL\left(q_A^{\tau_{A^c}},\mu_A^{\tau_{A^c}}\right)
\right].
\end{align*}
Apply Assumption~\ref{assume-scan-specific-factor} to $f=\widetilde\mu/\mu$ and use Lemma~\ref{lem:Ent-is-KL}. The left-hand side of \eqref{Ineq-ScanSpecificFactor} is $\gamma(p)\dKL(\widetilde\mu,\mu)$, while its right-hand side is
\[
C_p
\sum_{A\in\mathcal A}p_A
\E_{\widetilde\mu_{A^c}}
\left[
\dKL\left(\widetilde\mu_A^{\tau_{A^c}},\mu_A^{\tau_{A^c}}\right)
\right].
\]
The preceding inequality therefore proves the KL bound in \eqref{eq:approx-block-gibbs}; the Hellinger bound follows from \eqref{eq:H-vs-KL}.
\end{proof}

In particular, if \purple{numbers \(\varepsilon_A\ge0\), \(A\in\mathcal A\), satisfy}
\[
    \sup_{\tau_{A^c}\in[s]^{A^c}}
    \dKL\left(q_A^{\tau_{A^c}},\mu_A^{\tau_{A^c}}\right)
    \le \purple{\varepsilon_A^2},
    \qquad A\in\mathcal A,
\]
then
\begin{equation}\label{eq:approx-block-uniform}
    \dH^2(\widetilde\mu,\mu)
    \le
    \frac{C_p}{\gamma(p)}
    \sum_{A\in\mathcal A}p_A\purple{\varepsilon_A^2}.
\end{equation}

Importantly, the bound does not require the family of approximate conditionals $\{q_A^{\tau_{A^c}}\}_{A\in \mathcal{A}}$ to be compatible with any joint distribution.
Random-scan chains built from incompatible full conditional distributions are
sometimes called pseudo-Gibbs samplers; see \cite{TakabatakeAkaho2026} for a recent analysis.

Theorem~\ref{thm:main} treats perturbations of the stationary measure, and Theorem~\ref{thm:approx-block-gibbs} treats perturbations of the local updates. Section~\ref{SecWorkedExamples} uses both results and then gives a matching argument outside their common setting.

\section{Worked Examples}\label{SecWorkedExamples}

{\color{purple}
We organize the worked examples in three steps: first we give a class of Markov chains with no relationship to statistics, then statistical procedures that belong to the class, and finally consequences of our results for tuning of those statistical procedures. 

Section~\ref{subsec:Example1} treats high-temperature Ising models on a square box. It gives a class where the block-factorization results apply and compares the present argument with \cite{lin2025perturbation}. This example is pedagogical, and has no statistical content.

Section~\ref{subsec:Example2} then treats an Ising image posterior, a model used in image restoration since at least \cite{GemanGeman1984}. We analyze an augmented-state sampler and a simpler noisy sampler, and then specialize both bounds to a control-variate subsampling construction, with related approaches in \cite{HugginsAdamsBroderick2017,QuirozKohnVillaniTran2019}.

Section~\ref{subsec:Example3} follows the same order for weighted matchings. We first define exact and noisy single-edge Gibbs chains, then explain how they arise in Bayesian record linkage \cite{Sadinle2017Bayesian,quinn2023exploring,linacre2022splink}, and finally compare the resulting stationary bound with the direct one-step total-variation bound.
}

\subsection{The Ising Model on a Square Box}\label{subsec:Example1}

We begin by analyzing a perturbation of a simple Ising model on a square box.

\subsubsection{Notation}
We consider a growing sequence of grids indexed by their side length $L\in\mathbb N_+$. For fixed $L$, let
$G_L=(V_L,E_L)$
be the usual $L\times L$ box in $\mathbb Z^2$, with vertex set $V_L=[L]^2$ and nearest-neighbor edge set $E_L$; we drop the $L$ from the subscripts when not considering sequences of graphs.
Given coupling constants $J:E_L\to [0,\infty)$,
inverse temperature $\beta>0$, and external field
$h:V_L\to\mathbb R$,
we define the Ising Gibbs measure on
$\Omega=\{-1,+1\}^{V_L}$
by
\begin{equation} \label{EqDefIsingGibbs}
\mu(\sigma)
\propto
\exp\Bigg(
\beta\sum_{\{u,v\}\in E_L}J_{uv}\, \sigma(u)\,\sigma(v)
+
\sum_{v\in V_{L}} h_v\, \sigma(v)
\Bigg).
\end{equation}

In the notation of Section~\ref{Sec-Preliminaries}, the associated factor graph satisfies
\[
\Delta \leq 5,
\quad
\kappa=2,
\quad
|V_L|=L^2,
\quad
|E_L|=2L(L-1),
\quad
\Phi=\Phi_1 \cup \Phi_2,
\]
where
\[
\Phi_1=\{\phi_{uv}(\sigma)=\beta J_{uv}\,\sigma(u)\sigma(v):\, \{u,v\} \in E_L\}, \quad \Phi_2=\{\phi_v(\sigma)=h_v\,\sigma(v):\, v\in V_L\},
\]
with declared scopes $S[\phi_{uv}]=\{u,v\}$ and $S[\phi_v]=\{v\}$.

Let $\nu$ be another Ising Gibbs measure on the same graph with \purple{inverse temperature \(\beta'>0\) and external field \(h':V_L\to\mathbb R\)}. Then $\nu$ can be written as a perturbation of $\mu$ of the form
\[
\nu(\sigma)
\propto
\exp\Bigg(
\sum_{\phi\in\Phi}\phi(\sigma)+ g(\sigma)\Bigg),
\]
where the perturbation admits the local decomposition
\[
g(\sigma)
=
\sum_{\phi\in\Phi} g_\phi(\sigma),
\]
with
\[
g_{\phi}(\sigma)=
\begin{cases}
(\beta'-\beta)\,J_{uv}\,\sigma(u)\sigma(v),
& \phi=\phi_{uv}\in\Phi_1,
\\[0.3em]
(h'_v-h_v)\,\sigma(v),
& \phi=\phi_v\in\Phi_2.
\end{cases}
\]
Define the perturbation sizes by
\[
\delta_{\phi_{uv}}
=
\frac12\osc(g_{\phi_{uv}})
=
|\beta'-\beta|J_{uv},
\qquad \{u,v\} \in E_L,
\]
and
\[
\delta_{\phi_v}
=
\frac12\osc(g_{\phi_v})
=
|h_v'-h_v|,
\qquad v\in V_L.
\]

\subsubsection{Main Bound}
For neighboring vertices $\{u,v\} \in E_L$, define the Dobrushin influence coefficients
\[
c_{uv}
=
\sup_{\tau,\tau'}
\dTV\left(
\mu_{v}^{\tau_{V_L\setminus \{v\}}},
\mu_{v}^{\tau'_{V_L\setminus \{v\}}}
\right),
\]
where the supremum is taken over boundary conditions $\tau,\tau'$ differing only at the vertex $u$; that is, $\tau(y)=\tau'(y)$ for $y\neq u$.
The following is a standard estimate for the Dobrushin coefficients of the Ising model; we give the proof for completeness:

\begin{lemma}\label{lem:dobrushin-ising}
For the nearest-neighbor Ising model \eqref{EqDefIsingGibbs},
\[
c_{uv}\le \tanh(\beta J_{uv}), \qquad \{u,v\}\in E_L.
\]
\end{lemma}

\begin{proof}
Conditional on the boundary configuration outside $v$, the single-site conditional distribution at $v$ is
\[
\mu_v^{\tau_{V_L\setminus\{v\}}}(\sigma(v)=s)
=
\frac{
\exp\!\Big(
s\big(
h_v+\beta\sum_{w:\,\{w,v\}\in E_L}J_{vw}\tau(w)
\big)
\Big)
}{
2\cosh\!\Big(
h_v+\beta\sum_{w:\,\{w,v\}\in E_L}J_{vw}\tau(w)
\Big)
},
\qquad s\in\{-1,+1\}.
\]
Since $\tau$ and $\tau'$ differ only at $u$, the two conditional distributions differ only through the term $\beta J_{uv}\tau(u)$. The total variation distance is bounded by its maximum over the remaining effective field, which is attained when that field is zero. Therefore,
\[
\dTV
\left(
\mu_v^{\tau_{V_L\setminus\{v\}}},
\mu_v^{\tau'_{V_L\setminus\{v\}}}
\right)
\leq
\tanh(\beta J_{uv}).
\]
Taking the supremum over \(\tau,\tau'\) completes the proof.
\end{proof}

{\color{purple}
For the remainder of this subsection, set
\begin{equation*}\label{eq:ising-uniform-dobrushin-q}
q_L
=
\sup_{v\in V_L}
\sum_{u:\,\{u,v\}\in E_L}\tanh(\beta J_{uv}).
\end{equation*}
Lemma~\ref{lem:dobrushin-ising} shows that the actual Dobrushin row sum is at most \(q_L\).

\begin{lemma}\label{lem:ising-dobrushin-block-factorization}
Fix $L \in \mathbb{N}_+$.
Let \(G_L=(V_L,E_L)\) be a finite nearest-neighbor subgraph of the $L \times L$ square lattice  with maximum degree at most \(D\), and let \(\mu\) be an Ising measure with couplings \(J_{uv}\ge0\) and arbitrary external field. If $q_L<1$,
then \(\mu\) satisfies Assumptions~\ref{assume-SSM} and~\ref{assume-block-factor}, with constants depending only on \(D\) and \(q_L\).
\end{lemma}

\begin{proof}
Set \(C_{uv}=\tanh(\beta J_{uv})\) when \(\{u,v\}\in E\), and set \(C_{uv}=0\) otherwise. The row sums of \(C\) are at most \(q_L\), and the same matrix bounds the single-site influences after any collection of spins has been fixed. If \(q_L=0\), all interactions vanish and the assertion is immediate, so assume \(0<q_L<1\).

Fix \(A\subset B\subset V\), \(x\in\partial B\), and boundary conditions \(\tau,\tau'\) that differ only at \(x\). Write \(d=d(x,A)\), let \(S\) be the set of neighbors of \(x\) in \(B\), and let \(\rho=\mu_B^{\tau'_{B^c}}\). The change at \(x\) from $\tau'(x)$ to $\tau(x)$ is the tilt
\[
F(\sigma_S)
=
\exp\left\{
\beta\bigl(\tau(x)-\tau'(x)\bigr)
\sum_{y\in S}J_{xy}\sigma(y)
\right\}.
\]
Consequently, 
\[
\mu_B^{\tau_{B^c}}(\sigma_B)
=
\frac{F(\sigma_S)\rho(\sigma_B)}{\rho(F)}.
\]
For every \(\eta\in\{-1,+1\}^A\), taking the marginal on $A$ and conditioning under $\rho$ on
$\{\sigma_A=\eta\}$ therefore gives
\begin{align*}
\mu_A^{\tau_{B^c}}(\eta)
&= \sum_{\sigma_B : \sigma_A=\eta} \mu_B^{\tau_{B^c}}(\eta)\\
&= \frac{1}{\rho(F)} \sum_{\sigma_B : \sigma_A=\eta} F(\sigma_S)\rho(\sigma_B)\\
&= \frac{1}{\rho(F)} \, \rho(F \mathbf{1}_{\{\sigma_A=\eta\}})\\
&= \frac{1}{\rho(F)} \mu_A^{\tau'_{B^c}}(\eta)\, \rho(F\mid \sigma_A=\eta),
\end{align*}
and hence
\begin{equation}\label{eq:ising-ratio-via-boundary-tilt}
\frac{\mu_A^{\tau_{B^c}}(\eta)}{\mu_A^{\tau'_{B^c}}(\eta)}
=
\frac{\rho(F\mid\sigma_A=\eta)}{\rho(F)}.
\end{equation}
Since \(C_{xy}\le q_L\), each \(\beta J_{xy}\le\operatorname{arctanh}(q_L)\), and hence
\begin{equation}\label{eq:ising-boundary-tilt-range}
\frac{\sup F}{\inf F}
\le
\exp\{4D\operatorname{arctanh}(q_L)\}
=:
R_{D,q_L}.
\end{equation}
This already gives the required bound when \(d=1\).

Suppose now that \(d\ge2\). For configurations \(\eta,\eta'\) on \(A\), let \(\rho_S^\eta\) and \(\rho_S^{\eta'}\) be the conditional laws on \(S\). Corollary~2.6 of \cite{RebeschiniVanHandel2014Comparison}, applied after telescoping the boundary changes on \(A\), gives
\begin{align*}
\dTV(\rho_S^\eta,\rho_S^{\eta'})
&\le
\sum_{y\in S}\sum_{z\in A}\sum_{n\ge1}(C^n)_{yz}\\
&\le
D\sum_{n\ge d-1}q_L^n
\le
\frac{Dq_L^{d-1}}{1-q_L}.
\end{align*}
Indeed, a nonzero term \((C^n)_{yz}\) requires a path of length \(n\) from \(y\) to \(z\), and every such path has length at least \(d-1\). Since \(\rho(F)\) is a convex combination of the quantities \(\rho(F\mid\sigma_A=\eta')\), Equations~\eqref{eq:ising-ratio-via-boundary-tilt}--\eqref{eq:ising-boundary-tilt-range} imply
\[
\left|
\frac{\mu_A^{\tau_{B^c}}(\eta)}{\mu_A^{\tau'_{B^c}}(\eta)}-1
\right|
\le
(R_{D,q_L}-1)\frac{Dq_L^{d-1}}{1-q_L}.
\]
Taking \(m=-\log q_L\) and increasing \(M_{D,q_L}\) to cover the case \(d=1\), this is Assumption~\ref{assume-SSM}, with constants depending only on \(D\) and \(q_L\).

Finally, extend the couplings and fields to \(\mathbb Z^2\) by setting them equal to zero off \(G_L\). The same estimate holds uniformly for every finite region of this spatially non-homogeneous system. Theorem~2.3 and the extension in Section~2.2.4 of \cite{caputo2021block} therefore give Assumption~\ref{assume-block-factor}, with a constant depending only on \(D\) and \(q_L\).
\end{proof}
}

In particular, in the homogeneous case \(J_{uv}\equiv 1\), we have
\[ q_L\le 4\tanh(\beta). \]
Hence, whenever \(4\tanh(\beta)<1\), the uniform Dobrushin condition \(q_L\le q<1\) holds for any choice of \(q\in [4\tanh(\beta),1)\).

\begin{theorem}\label{thm:torus-ising}
Assume there exists $q<1$ such that
\begin{equation}\label{eq:torus-dobrushin-uniform}
q_L\le q
\end{equation}
for all $L$. Then there is a constant $C_{q}<\infty$, independent of $L$, such that
\begin{equation}\label{eq:torus-ising-bound}
\dH^2(\nu,\mu)
\le
C_{q}
\left(
\sum_{\{u,v\}\in E_L}|\beta'-\beta|^2J_{uv}^2
+
\sum_{v\in V_L}|h_v'-h_v|^2
\right).
\end{equation}

Furthermore, if
\begin{equation}\label{eq:torus-ising-l2-perturbation}
\sum_{\{u,v\}\in E_L}|\beta'-\beta|^2J_{uv}^2
+
\sum_{v\in V_L}|h_v'-h_v|^2
\longrightarrow 0,
\end{equation}
then
\[
\dH(\nu,\mu)\to0
\qquad\text{as }L\to\infty.
\]
In particular, \eqref{eq:torus-ising-l2-perturbation} holds if
\begin{equation}\label{eq:torus-ising-small-perturbation}
\left(\max_{v\in V_L}|h_v'-h_v|\right)
\vee
\left(\sqrt{2}|\beta'-\beta|\max_{\{u,v\}\in E_L}J_{uv}\right)
=
o(L^{-1}).
\end{equation}
\end{theorem}

\begin{remark}[Comparison to standard perturbation estimates]
The left-hand side of Inequality~\eqref{eq:torus-ising-small-perturbation} controls, up to constant factors, the worst-case total variation distance between the usual single-site Gibbs kernels targeting $\mu$ and $\nu$. Thus, our result says that $\mu, \nu$ are close as long as we have strong spatial mixing (Inequality~\eqref{eq:torus-dobrushin-uniform}) and the corresponding one-step kernel error is less than roughly $\frac{1}{L}$. Naively using classical perturbation bounds such as \cite{Mitrophanov2005} requires this error to be less than roughly $\frac{1}{L^{2}}$, which is (up to $O(\log(L))$ factors) the inverse of the mixing time of the single-site Gibbs sampler at sufficiently high temperature (that is, for $\beta$ sufficiently small).
\end{remark}

\begin{proof}[Proof of Theorem~\ref{thm:torus-ising}]

\purple{By Lemma~\ref{lem:ising-dobrushin-block-factorization}, Inequality~\eqref{eq:torus-dobrushin-uniform} implies Assumption~\ref{assume-block-factor}, with a constant \(C_1\) depending only on \(q\), since the square grid has maximum degree at most four. In particular, \(C_1\) is independent of \(L\) and of the external field. Assumption~\ref{assume-approx-tensor} therefore holds with \(C_{\mathrm{AT}}=C_1\).}

Applying Theorem~\ref{thm:main} and using $\Delta\le5$ and $\kappa=2$ gives
\begin{align}
\dH^2(\nu,\mu)
&\le
\frac{C_1\Delta\kappa}{2}
\sum_{\phi\in\Phi}\delta_\phi^2 \notag\\
&\le
5C_1
\left(
\sum_{\{u,v\}\in E_L}|\beta'-\beta|^2J_{uv}^2
+
\sum_{v\in V_L}|h_v'-h_v|^2
\right).
\label{eq:torus-ising-proof-bound}
\end{align}
This proves \eqref{eq:torus-ising-bound} with $C_{q}=5C_1$, and \eqref{eq:torus-ising-l2-perturbation} gives the first asymptotic conclusion. For the final assertion, let the left-hand side of \eqref{eq:torus-ising-small-perturbation} be $a_L$. Since $|V_L|=L^2$ and $|E_L|=2L(L-1)\le2L^2$,
\[
\sum_{\{u,v\}\in E_L}|\beta'-\beta|^2J_{uv}^2
+
\sum_{v\in V_L}|h_v'-h_v|^2
\le
2L^2a_L^2
=
o(1).
\]
Thus \eqref{eq:torus-ising-small-perturbation} implies \eqref{eq:torus-ising-l2-perturbation}.

\end{proof}

\subsubsection[Comparison to Earlier Bounds]{Comparison to Bounds in \cite{lin2025perturbation}}\label{subsubsec:comparison-previous}

We now give two partial comparisons with \cite{lin2025perturbation}. In the common setting of local Gibbs perturbations, Theorem~\ref{thm:main} is stronger: it gives a direct square-sum estimate without the sequential set decomposition used there. For the more general perturbations treated in \cite{lin2025perturbation}, the inputs differ, so only partial comparisons are possible. Proposition~\ref{prop:hellinger-block-bound} compares the aggregation steps more broadly, and we then examine the standard uniform verification of the assumptions in \cite{lin2025perturbation} on the square grid. 

Recall that the key ingredient in \cite{lin2025perturbation} is the square-Hellinger subadditivity inequality
\begin{equation}\label{Ineq-Subadd-Hellinger}
\dH^2(\nu,\mu)
\leq
\sum_{j=1}^{\ell}
\dH^2\bigl(\nu_{S_j\cup \Pi_j},\mu_{S_j\cup \Pi_j}\bigr),
\end{equation}
which applies whenever the two measures admit a common factorization of the form
\begin{equation}\label{Bayesian-Factor}
\mu(\sigma)
=
\prod_{j=1}^{\ell}
\mu(\sigma_{S_j}\mid\sigma_{\Pi_j}),
\qquad
\nu(\sigma)
=
\prod_{j=1}^{\ell}
\nu(\sigma_{S_j}\mid\sigma_{\Pi_j}),
\end{equation}
where $\{S_j\}_{j=1}^{\ell}$ forms a partition of $V$ and
\begin{equation}\label{EqContainmentFactor}
\Pi_1=\emptyset,
\qquad
\Pi_j \subseteq \bigcup_{i=1}^{j-1} S_i.
\end{equation}

Under block factorization, conditional Hellinger distances can instead be combined as follows.

\begin{proposition}\label{prop:hellinger-block-bound}
Suppose that Assumption~\ref{assume-block-factor} holds and that $\nu\ll\mu$. Then, for any family of blocks $\mathcal A\subseteq 2^V$ and any choice of nonnegative weights $\alpha=\{\alpha_A\}_{A\in\mathcal A}$ with $\gamma(\alpha)>0$,
\begin{equation*}\label{eq:hellinger-block-bound}
\dH^2(\nu,\mu)
\leq
\frac{2C_{\mathrm{BF}}}{\gamma(\alpha)}
\sum_{A\in\mathcal A}
\alpha_A
\E_{\nu_{A^c}}
\left[
\dH^2\left(
\nu_A^{\tau_{A^c}},
\mu_A^{\tau_{A^c}}
\right)
\right],
\end{equation*}
where $C_{\mathrm{BF}}$ is the constant in Assumption~\ref{assume-block-factor}. The conditional distribution of $\nu$ may be chosen arbitrarily when the conditioning event has probability zero.
\end{proposition}

The proof is given in Appendix~\ref{app:comparison-previous}. Set $A_j=S_j\cup\Pi_j$ for each $j\in [\ell]$, with unit weight. Since the sets $S_j$ partition $V$ and $\gamma(\alpha)\geq 1$, Proposition~\ref{prop:hellinger-block-bound} gives
\begin{equation}\label{eq:hellinger-sequential-blocks}
\dH^2(\nu,\mu)
\leq
2C_{\mathrm{BF}}
\sum_{j=1}^{\ell}
\E_{\nu_{A_j^c}}
\left[
\dH^2\left(
\nu_{A_j}^{\tau_{A_j^c}},
\mu_{A_j}^{\tau_{A_j^c}}
\right)
\right].
\end{equation}
The terms in \eqref{Ineq-Subadd-Hellinger} are marginal Hellinger distances, whereas those in \eqref{eq:hellinger-sequential-blocks} are full-boundary conditional Hellinger distances, so there is no termwise comparison in general. The point is instead that, at $r=1$, when $B_1(A_j)=A_j$ in the notation of \cite{lin2025perturbation}, Assumptions 2 and 3 there directly bound every conditional Hellinger distance in \eqref{eq:hellinger-sequential-blocks} by $C_2C_3f(1)\varepsilon$, where $\varepsilon$ is the one-step kernel error used in that paper. Therefore
\[
\dH(\nu,\mu)
\leq
\sqrt{2C_{\mathrm{BF}}\ell}\,C_2C_3f(1)\varepsilon.
\]
Assumption 4 of \cite{lin2025perturbation} makes the right-hand side at most $\sqrt{2C_{\mathrm{BF}}}\,C_2C_3C_4f(1)/f(\log\ell)$, so Proposition~\ref{prop:hellinger-block-bound} recovers the perturbative scale used there. More generally, under block factorization it only requires $\sqrt{\ell}\,\varepsilon\to 0$, and therefore removes the extra $f(\log\ell)$ loss caused by taking $r$ of order $\log\ell$ in \cite{lin2025perturbation}. This comparison retains the propagation-of-perturbations assumption from \cite{lin2025perturbation} and isolates the global aggregation step. For local Gibbs perturbations, Theorem~\ref{thm:main} goes further by bounding the stationary error directly from the local potential errors.

The standard uniform verification discussed in \cite{lin2025perturbation} is most effective when the sets $S_j \cup \Pi_j$ remain uniformly small. That turns out to be impossible for this model. For the remainder of this subsection, assume that $J_{uv}>0$ for every $\{u,v\}\in E_L$. We have the following lower bound on the sizes of the sets appearing in the factorization:

\begin{lemma} \label{LemmaPreDoesntWork}
Let $\{S_{j}^{(L)}, \Pi_{j}^{(L)}\}$ be a sequence of sets satisfying the $\mu$-factorization in \eqref{Bayesian-Factor} and \eqref{EqContainmentFactor} for the Ising model on $G_{L}$. Then
\[
\max_j |S_j^{(L)}\cup\Pi_j^{(L)}| \gtrsim L.
\]
\end{lemma}
\begin{proof}
The proof relies on an object called the ``tree width" of a graph, first defined in \cite{ROBERTSON1986309}. Since the concept of tree width is not central to the paper, we will treat it as a well-studied black-box function $\operatorname{tw}$ sending graphs to nonnegative integers.

Lemma~\ref{lem:factorization-treewidth} gives
\[
\operatorname{tw}(G_L)
\leq
\max_j |S_j^{(L)}\cup\Pi_j^{(L)}|-1.
\]
For $L\geq 2$, the $L\times L$ grid has treewidth $L$ \cite[Corollary~89]{bodlaender1998partial}. Combining these two bounds gives the conclusion.
\end{proof}

Lemma~\ref{LemmaPreDoesntWork} shows that at least one of the sets in \eqref{Ineq-Subadd-Hellinger} must grow with $L$. To compare scales, consider the standard uniform verification of Assumption~3 described in Remark~4 of \cite{lin2025perturbation}. Let $A_j=S_j\cup\Pi_j$, write
\[
M_L=\max_j |A_j|,
\qquad
\varepsilon_L=\sup_S d(Q|_S,K|_S),
\]
where the supremum is over the restricted chains used in that verification. If the inverse relaxation rate on a region $S$ is of order $|S|$, the usual perturbation estimate gives, schematically,
\begin{equation}
\dH(\nu,\mu)
\lesssim
\sqrt{\ell}\left(
e^{-mr}+\max_j |B_r(A_j)|\,\varepsilon_L
\right).
\label{eq:torus-factorization-bound}
\end{equation}
This is the standard uniform route proposed in \cite{lin2025perturbation}, not a lower bound on every possible use of that result. Choosing $r$ of order $\log\ell$ controls the first term. Even before enlarging $A_j$ to $B_r(A_j)$, the coefficient of $\varepsilon_L$ is at least $\sqrt{\ell}M_L$. Since the sets $S_j$ partition the $L^2$ vertices,
\[
\ell M_L\geq \sum_{j=1}^{\ell}|S_j|=L^2,
\]
and Lemma~\ref{LemmaPreDoesntWork} gives $M_L\gtrsim L$. Hence
\[
\sqrt{\ell}M_L
=
\sqrt{(\ell M_L)M_L}
\gtrsim
L^{3/2}.
\]
For the single-site Gibbs perturbations considered here, with the same random-scan normalization, $\varepsilon_L$ is comparable up to constants to the perturbation scale on the left-hand side of \eqref{eq:torus-ising-small-perturbation}. By contrast, Theorem~\ref{thm:torus-ising} gives a bound with coefficient of order $L$ on this scale. Thus this standard uniform verification of \cite{lin2025perturbation} loses an unbounded factor, at least $L^{1/2}$, on the square grid. A nonuniform or model-specific verification could give a better comparison.

{\color{purple}Theorem~\ref{thm:torus-ising} gives the perturbation scale for this class. Proposition~\ref{prop:hellinger-block-bound} recovers the square-root aggregation directly from the local conditional estimates in \cite{lin2025perturbation}, while Lemma~\ref{LemmaPreDoesntWork} explains why the standard uniform verification proposed there can lose an unbounded factor on the square grid. Section~\ref{subsec:Example2} applies Theorems~\ref{thm:main} and~\ref{thm:approx-block-gibbs} to image sampling.}

\subsection{Using Ising Priors for Image Observations}\label{subsec:Example2}

We next consider a Bayesian model for image restoration with Ising priors. Models of this sort for Bayesian image restoration go back at least to Geman and Geman \cite{GemanGeman1984}. In this section, we assume that there have been many noisy observations per pixel - the number $n_{v}$ in Equation \eqref{EqDefNObs} corresponds to the number of observations of pixel $v$, and of course the case that $n_{v} = n$ does not depend on $v$ is typical.

We study two closely-related MCMC algorithms for this problem. We first consider a sampler on an augmented state space, and then a simpler version without the auxiliary-variable Metropolis--Hastings update. In order to obtain reasonable estimates, we combine subsampling with precomputed control variates; related data-reduction approaches appear in \cite{QuirozKohnVillaniTran2019,HugginsAdamsBroderick2017}.

\subsubsection{Notation}
Let \(G=(V,E)\) be a finite nearest-neighbor subgraph of \(\mathbb Z^2\) and let \(\Omega=\{-1,+1\}^{V}\).  For a prior distribution, we use the Ising model in \eqref{EqDefIsingGibbs} with coupling constants \(J:E\to[0,\infty)\), inverse temperature \(\beta>0\), and external field \(h:V\to\mathbb R\).
For each vertex \(v\in V\), let
\begin{equation} \label{EqDefNObs}
Z_v=\{z_{v,1},\dots,z_{v,n_v}\},
\qquad n_v\ge2,
\end{equation}
and let \(\ell_v(\sigma(v);z_{v,j})\) be the log-likelihood of an observed data point $z_{v,j}$ given a ``true'' image $\sigma \in \Omega$.\footnote{For \(z\in\{-1,+1\}\), a prototypical log-likelihood for this situation is \(\exp(\ell_v(\sigma(v);z))=p\mathbf 1_{\{\sigma(v)=z\}}+(1-p)\mathbf 1_{\{\sigma(v)\neq z\}}\) for some \(0.5<p<1\).  This says that there is a true image \(\sigma\), and each observation is flipped with probability \(1-p\).} \purple{Assume that these log-likelihoods are finite for every observed datum.}
Conditional on \(\sigma\), assume the observations are independent.  Then the posterior distribution on \(\Omega\) is
\begin{equation*}\label{eq:ising-posterior-before-field}
\mu(\sigma\mid Z)
\propto
\exp\Bigg(
\beta\sum_{\{u,v\}\in E}J_{uv}\,\sigma(u)\sigma(v)
+
\sum_{v\in V}h_v\,\sigma(v)
+
\sum_{v\in V}\sum_{j=1}^{n_v}
\ell_v(\sigma(v);z_{v,j})
\Bigg).
\end{equation*}

Since \(\sigma(v)\in\{-1,+1\}\), each local likelihood contribution can be written as
a constant plus a linear term in \(\sigma(v)\).  Define
\[
C_v(Z_v)
=
\frac12\sum_{j=1}^{n_v}
\left\{
\ell_v(+1;z_{v,j})+\ell_v(-1;z_{v,j})
\right\},
\]
\[
r_v(z)
=
\frac12\left\{
\ell_v(+1;z)-\ell_v(-1;z)
\right\},
\qquad
H_v(Z_v)
=
\sum_{j=1}^{n_v} r_v(z_{v,j}).
\]
Then
\[
\sum_{j=1}^{n_v}\ell_v(\sigma(v);z_{v,j})
=
C_v(Z_v)+H_v(Z_v)\sigma(v).
\]
Dropping the terms \(C_v(Z_v)\), which do not depend on \(\sigma\), gives
\begin{equation}\label{eq:ising-posterior}
\mu(\sigma\mid Z)
\propto
\exp\Bigg(
\beta\sum_{\{u,v\}\in E}J_{uv}\,\sigma(u)\sigma(v)
+
\sum_{v\in V}\big(h_v+H_v(Z_v)\big)\sigma(v)
\Bigg).
\end{equation}
In \eqref{eq:ising-posterior}, the data enter only through the second term in the exponent (corresponding to individual vertices), not through the first term (corresponding to edges). Thus \purple{Lemma~\ref{lem:ising-dobrushin-block-factorization} checks the spatial-mixing and block-factorization assumptions uniformly in the observed dataset \(Z\) whenever the coupling bound in \eqref{eq:ising-observation-dobrushin-q} is less than one.}

We next set notation for the exact Markov chain.  Let \(\psi(x)=e^x/(e^x+e^{-x})\).  For \(v\in V\), write the exact single-site conditional field as
\begin{equation}\label{eq:exact-singlesite-field}
A_v(\sigma)=\beta\sum_{u:\,\{u,v\}\in E}J_{uv}\sigma(u)+h_v+H_v(Z_v),
\end{equation}
so that
\begin{equation*}
\mu\big(\sigma(v)=+1\mid \sigma_{V\setminus\{v\}},Z\big)
=
\psi(A_v(\sigma)).
\end{equation*}
Algorithm~\ref{alg:exact-gibbs} describes one step of the exact random-scan Gibbs kernel \(K\), whose invariant distribution is \(\mu(\cdot\mid Z)\).

\begin{algorithm}[H]
\caption{One step of the exact random-scan Gibbs sampler}
\label{alg:exact-gibbs}
\begin{algorithmic}[1]
\Require Current configuration \(\sigma\in\{-1,+1\}^V\), data \(Z\).
\State Sample a vertex \(v\sim \mathrm{Unif}(V)\).
\State Compute $H_v(Z_v)$ and the exact field $A_v(\sigma)$ according to \eqref{eq:exact-singlesite-field}.
\State Sample \(B\sim \mathrm{Bernoulli}(\psi(A_v(\sigma)))\), and set
\[
\sigma'(v)=
\begin{cases}
+1, & B=1,\\
-1, & B=0
\end{cases}
\]
and
\[
\sigma'(u)=\sigma(u),\quad u\neq v.
\]
\State \Return \(\sigma'\).
\end{algorithmic}
\end{algorithm}

For both approximate samplers, let \(\xi_v\) be an auxiliary variable with distribution \(Q_v\), and let
\[
\widehat H_v=\widehat H_v(Z_v,\xi_v),
\qquad
e_v(\xi_v)=\widehat H_v(Z_v,\xi_v)-H_v(Z_v).
\]
All expectations in the remainder of this subsection are conditional on \(Z\). Assume that the auxiliary variables are independent under \(Q=\prod_{v\in V}Q_v\).
\purple{Assume also that}
\(\E[\exp(\widehat H_v(Z_v,\xi_v))]\) and
\(\E[\exp(-\widehat H_v(Z_v,\xi_v))]\) are finite. On the augmented state space, define
\begin{equation}\label{eq:augmented-state-target}
\widetilde\pi(\sigma,d\xi\mid Z)
\propto
\exp\Bigg(
\beta\sum_{\{u,v\}\in E}J_{uv}\,\sigma(u)\sigma(v)
+
\sum_{v\in V}\big(h_v+\widehat H_v(Z_v,\xi_v)\big)\sigma(v)
\Bigg)
Q(d\xi).
\end{equation}

Integrating \eqref{eq:augmented-state-target} over the auxiliary variables yields
\[
\begin{aligned}
\widetilde\mu(\sigma\mid Z)
&\propto
\exp\Bigg(
\beta\sum_{\{u,v\}\in E}J_{uv}\sigma(u)\sigma(v)
+
\sum_{v\in V}h_v\sigma(v)
\Bigg)\\
&\qquad {}\purple{\times}
\prod_{v\in V}
\E_{Q_v}\!\left[
    \exp\left(\widehat H_v(Z_v,\xi_v)\sigma(v)\right)
\right].
\end{aligned}
\]
For each \(v\in V\), set
\[
\purple{a_{v,+}}=\E_{Q_v}\left[\exp\left(\widehat H_v(Z_v,\xi_v)\right)\right],
\qquad
\purple{a_{v,-}}=\E_{Q_v}\left[\exp\left(-\widehat H_v(Z_v,\xi_v)\right)\right],
\]
and
\begin{equation}\label{eq:effective-ising-field}
\widetilde H_v(Z_v)
=
\frac12\log
\frac{\purple{a_{v,+}}}{\purple{a_{v,-}}}.
\end{equation}
Since \(\sigma(v)\in\{-1,+1\}\), we have
\[
\E_{Q_v}\!\left[
    \exp\left(\widehat H_v(Z_v,\xi_v)\sigma(v)\right)
\right]
=
\sqrt{\purple{a_{v,+}a_{v,-}}}\,
\exp\left(\widetilde H_v(Z_v)\sigma(v)\right).
\]
The factor \(\prod_{v\in V}\sqrt{\purple{a_{v,+}a_{v,-}}}\) does not depend on \(\sigma\), and is
therefore absorbed into the normalizing constant. Hence the marginal distribution
of \(\widetilde\pi\) on \(\Omega\) is
\begin{equation}\label{eq:augmented-state-marginal}
\widetilde\mu(\sigma\mid Z)
\propto
\exp\Bigg(
\beta\sum_{\{u,v\}\in E}J_{uv}\,\sigma(u)\sigma(v)
+
\sum_{v\in V}\big(h_v+\widetilde H_v(Z_v)\big)\sigma(v)
\Bigg).
\end{equation}

Each update of the augmented sampler consists of two stages.  First, conditional on the current auxiliary variable \(\xi_v\), it performs a Gibbs update of the spin at $v$, targeting $\widetilde{\pi}(\sigma(v)\mid \purple{\sigma_{V \setminus \{v\}}},\xi,Z)$. The corresponding approximate single-site field is
\begin{equation}\label{eq:approximate-singlesite-field}
\widehat A_v(\sigma, \xi_v)
=
\beta\sum_{u:\,\{u,v\}\in E}J_{uv}\sigma(u)
+h_v+\widehat H_v(Z_v,\xi_v)
=
A_v(\sigma)+e_v(\xi_v).
\end{equation}
Equivalently,
\[
\widetilde\pi\bigl(\sigma(v)=+1
\mid \sigma_{V\setminus\{v\}},\xi,Z\bigr)
=
\psi\bigl(\widehat A_v(\sigma, \xi_v)\bigr).
\]
After this spin update, the sampler refreshes the auxiliary variable \(\xi_v\) by a Metropolis--Hastings step targeting \(\widetilde\pi(d\xi_v\mid \sigma',\xi_{V\setminus\{v\}},Z)\).
Thus the resulting random-scan sampler leaves \(\widetilde\pi\) invariant, and its \(\sigma\)-marginal is \(\widetilde\mu\) in \eqref{eq:augmented-state-marginal}.
One step of the sampler is given in Algorithm~\ref{alg:noisy-gibbs}.

\begin{algorithm}[H]
\caption{One step of the random-scan sampler on \((\sigma,\xi)\)}
\label{alg:noisy-gibbs}
\begin{algorithmic}[1]
\Require Current state \((\sigma,\xi)\), data \(Z\).
\State Sample \(v\sim \mathrm{Unif}(V)\).
\State Compute the approximate field \(\widehat A_v(\sigma,\xi_v)\) according to \eqref{eq:approximate-singlesite-field}.
\State Sample \(B\sim \mathrm{Bernoulli}(\psi(\widehat A_v(\sigma,\xi_v)))\), and set
\[
\sigma'(v)=
\begin{cases}
+1, & B=1,\\
-1, & B=0,
\end{cases}
\qquad
\sigma'(u)=\sigma(u),\quad u\neq v.
\]
\State Independently sample \(\xi_v^\star\sim Q_v\), and set
\[
\alpha_v
=
1\wedge
\exp\!\left(
\sigma'(v)
\big(
\widehat H_v(Z_v,\xi_v^\star)-\widehat H_v(Z_v,\xi_v)
\big)
\right).
\]
\State With probability \(\alpha_v\), set \(\xi'_v=\xi_v^\star\); otherwise set \(\xi'_v=\xi_v\).  Set \(\xi'_u=\xi_u\) for \(u\neq v\).
\State \Return \((\sigma',\xi')\).
\end{algorithmic}
\end{algorithm}

\subsubsection{\purple{Sampler on the augmented state space}} \label{Sec42MainBound}
Set
\begin{equation}\label{eq:ising-observation-dobrushin-q}
q:=\sup_{v\in V}\sum_{u:\,\{u,v\}\in E}\tanh(\beta J_{uv}).
\end{equation}

\begin{assumption}\label{ass:e_v}
For each \(v\in V\), there exists \purple{\(0\le\vartheta_v^2<\infty\)} such that
\[
\E[e_v]=0,
\qquad
\E\big[\exp(t e_v)\big]
\le
\exp\left(\frac{t^2\vartheta_v^2}{2}\right),
\qquad t\in\mathbb R.
\]
\end{assumption}

Further discussion of this assumption for sampling-without-replacement estimators is provided in Section~\ref{SecCVSubs}.

\begin{theorem}\label{thm:noisy-gibbs-cv}
Suppose that \(G\) has maximum degree at most \(D\), that \(q<1\), and that Assumption~\ref{ass:e_v} holds. Then there is a constant \(C_{D,q}<\infty\), independent of \(|V|\), \(Z\), and the field estimators, such that
\begin{equation}\label{eq:noisy-gibbs-main}
\dH^2\bigl(\widetilde\mu(\cdot\mid Z),\mu(\cdot\mid Z)\bigr)
\le
C_{D,q}\sum_{v\in V}\vartheta_v^4.
\end{equation}
\end{theorem}

\begin{proof}
By \eqref{eq:effective-ising-field},
\[
\widetilde H_v(Z_v)-H_v(Z_v)
=
\frac12
\left(
\log\E[\exp(e_v)]-\log\E[\exp(-e_v)]
\right).
\]
By Jensen's inequality and the assumption \(\E[e_v]=0\), both logarithms are nonnegative. \purple{The second part of Assumption~\ref{ass:e_v} bounds each logarithm by \(\vartheta_v^2/2\), so their difference has absolute value at most \(\vartheta_v^2/2\). Therefore}
\[
\big|\widetilde H_v(Z_v)-H_v(Z_v)\big|
\le
\purple{\frac{\vartheta_v^2}{4}}.
\]
\purple{Lemma~\ref{lem:ising-dobrushin-block-factorization} shows that the condition \(q<1\) in \eqref{eq:ising-observation-dobrushin-q} implies Assumption~\ref{assume-block-factor}, with a constant depending only on \(D\) and \(q\), and hence Assumption~\ref{assume-approx-tensor} with the same constant.}

The measures \(\widetilde\mu(\cdot\mid Z)\) and \(\mu(\cdot\mid Z)\) differ only through the change in external field
\[
g(\sigma)
=
\sum_{v\in V}
\big(\widetilde H_v(Z_v)-H_v(Z_v)\big)\sigma(v).
\]
Applying Theorem~\ref{thm:main} to this perturbation with
\[
\delta_v
=
\frac12\osc\left(\bigl(\widetilde H_v(Z_v)-H_v(Z_v)\bigr)\sigma(v)\right)
=
\big|\widetilde H_v(Z_v)-H_v(Z_v)\big|
\le
\purple{\frac{\vartheta_v^2}{4}}
\]
proves \eqref{eq:noisy-gibbs-main}.
\end{proof}

\subsubsection{\purple{Naive Noisy Algorithm}}\label{subsec:fresh-field-estimates}

{\color{purple}We now use the same subsample-based field estimates with a simpler naive sampler.} The auxiliary variable \(\xi_v\) is used only for the current spin update, rather than kept as part of the state and updated by the Metropolis--Hastings step in Algorithm~\ref{alg:noisy-gibbs}. Thus the resulting chain does not generally have a stationary distribution of the form \eqref{Perturbed-Measure}. One step is given in Algorithm~\ref{alg:fresh-field-gibbs}.

\begin{algorithm}[H]
\caption{One step of the naive random-scan sampler}
\label{alg:fresh-field-gibbs}
\begin{algorithmic}[1]
\Require Current configuration \(\sigma\in\{-1,+1\}^V\), data \(Z\).
\State Sample \(v\sim\mathrm{Unif}(V)\).
\State Independently sample \(\xi_v\sim Q_v\).
\State Compute \(\widehat A_v(\sigma,\xi_v)\) according to \eqref{eq:approximate-singlesite-field}.
\State Sample \(B\sim\mathrm{Bernoulli}(\psi(\widehat A_v(\sigma,\xi_v)))\), and set
\[
\sigma'(v)=
\begin{cases}
+1, & B=1,\\
-1, & B=0,
\end{cases}
\qquad
\sigma'(u)=\sigma(u),\quad u\neq v.
\]
\State \Return \(\sigma'\).
\end{algorithmic}
\end{algorithm}

For \(v\in V\) and
\(\tau_{\{v\}^c} \in\{-1,+1\}^{V\setminus\{v\}}\), let \(A_v(\tau_{\{v\}^c})\) denote the value
of the field in \eqref{eq:exact-singlesite-field} under the boundary condition \(\tau_{\{v\}^c}\). After averaging over the new draw \(\xi_v\sim Q_v\), the single-site
update distribution is
\begin{equation}\label{eq:fresh-field-conditional}
    \overline q_v^{\tau_{\{v\}^c}}(+1)
    =
    \E\left[\psi\bigl(A_v(\tau_{\{v\}^c})+e_v\bigr)\right],
    \qquad
    \overline q_v^{\tau_{\{v\}^c}}(-1)
    =
    1-\overline q_v^{\tau_{\{v\}^c}}(+1).
\end{equation}
Let \(\overline K\) denote the transition kernel in
Algorithm~\ref{alg:fresh-field-gibbs}, which can therefore be written as
\[
\overline K(\tau,\sigma)
=
\frac{1}{|V|}
\sum_{v\in V}
\mathbf{1}\{\sigma_{\{v\}^c}=\tau_{\{v\}^c}\}
\overline q_v^{\tau_{\{v\}^c}}(\sigma(v)).
\]
Since all single-site update probabilities in \eqref{eq:fresh-field-conditional} are strictly positive, \(\overline K\) is irreducible and aperiodic on the finite state space \(\Omega\). Hence it has a unique invariant distribution, which we denote by \(\overline\mu\).

We then have the following error estimate, which is essentially the same as Theorem \ref{thm:noisy-gibbs-cv}:

\begin{theorem}\label{thm:fresh-field-gibbs}
Suppose that \(G\) has maximum degree at most \(D\), that \(q<1\), and that
Assumption~\ref{ass:e_v} holds. Then there is a constant
\(C_{D,q}<\infty\), independent of \(|V|\), \(Z\), \purple{and the field estimators},
such that
\begin{equation}\label{eq:fresh-field-main}
    \dH^2\bigl(\overline\mu(\cdot\mid Z),\mu(\cdot\mid Z)\bigr)
    \le
    C_{D,q}\sum_{v\in V}\vartheta_v^4.
\end{equation}
\end{theorem}

\begin{proof}
Fix \(v\in V\) and a boundary condition \(\tau_{\{v\}^c}\), and write
\[
    \overline p=\overline q_v^{\tau_{\{v\}^c}}(+1),
    \qquad
    \overline a=\frac12\log\frac{\overline p}{1-\overline p}.
\]
Then
\[
\overline q_v^{\tau_{\{v\}^c}} (s)=\psi (\overline{a} s), \qquad s\in \{+1,-1\}.
\]
Set \(X=e^{2e_v}\), $a=A_v(\tau_{\{v\}^c})$ and \(r=e^{2a}\). By
\eqref{eq:fresh-field-conditional},
\begin{equation*}\label{eq:fresh-field-effective-odds}
    e^{2(\overline a-a)}
    =
    \frac{\E[X/(1+rX)]}{\E[1/(1+rX)]}.
\end{equation*}
Since \(x\mapsto x\) is increasing and \(x\mapsto(1+rx)^{-1}\) is
decreasing,
\[
    \E\left[\frac{X}{1+rX}\right]
    \le
    \E[X]\E\left[\frac{1}{1+rX}\right].
\]
Similarly, \(x\mapsto x^{-1}\) is decreasing and
\(x\mapsto x/(1+rx)\) is increasing, so
\[
    \E\left[\frac{1}{1+rX}\right]
    \le
    \E[X^{-1}]\E\left[\frac{X}{1+rX}\right].
\]
These imply that
\[
(E[X^{-1}])^{-1} \le e^{2(\overline a-a)} \leq \E [X].
\]
Assumption~\ref{ass:e_v}, applied with \(t=\pm2\), therefore gives
\[
    e^{-2 \vartheta_v^2}
    \le
    e^{2(\overline a-a)}
    \le
    e^{2 \vartheta_v^2},
\]
and hence
\begin{equation}\label{eq:fresh-field-effective-field}
    |\overline a-a|\le \vartheta_v^2.
\end{equation}
As distributions on \(\{-1,+1\}\), \(\overline q_v^{\tau_{\{v\}^c}}\) is obtained from
\(\mu_v^{\tau_{\{v\}^c}}\) by changing the one-site field from \(a\) to \(\overline a\), with \(g_v(s)=(\overline a-a)s\). Lemma~\ref{lem:tilt}
and \eqref{eq:fresh-field-effective-field} therefore give
\begin{equation}\label{eq:fresh-field-local-kl}
    \dKL\left(\overline q_v^{\tau_{\{v\}^c}},\mu_v^{\tau_{\{v\}^c}}\right)
    \le
    \frac12 |\overline a-a|^2
    \le
    \frac12 \vartheta_v^4.
\end{equation}
\purple{Lemma~\ref{lem:ising-dobrushin-block-factorization} shows that \(q<1\) implies Assumption~\ref{assume-block-factor}, with a constant depending only on \(D\) and \(q\), and hence Assumption~\ref{assume-scan-specific-factor} for the scan below with the same constant.} Apply Theorem~\ref{thm:approx-block-gibbs} with the
singleton blocks \(A=\{v\}\) and scan probabilities
\(p_{\{v\}}=|V|^{-1}\). Since \(\gamma(p)=|V|^{-1}\),
\eqref{eq:fresh-field-local-kl} proves \eqref{eq:fresh-field-main}.
\end{proof}

\subsubsection{Control-variate subsampling} \label{SecCVSubs}
{\color{purple}We now give a concrete subsampling construction for both algorithms and show how to check Assumption~\ref{ass:e_v}.}

For each \(v\in V\), choose a control-variate function \(c_v\) intended to approximate \(r_v\), and define residuals
\[
R_{v,j}
=
r_v(z_{v,j})-c_v(z_{v,j}),
\qquad
\overline R_v
=
\frac1{n_v}\sum_{j=1}^{n_v}R_{v,j}.
\]
Let \(S_v\subseteq\{1,\ldots,n_v\}\) be sampled uniformly without replacement with
\(|S_v|=m_v\) and \(1\le m_v\le n_v\). We approximate \(H_v(Z_v)\) by
\begin{equation}
\label{eq:Hv-hat-special}
\widehat H_v(Z_v,S_v)
=
\sum_{j=1}^{n_v}c_v(z_{v,j})
+
\frac{n_v}{m_v}\sum_{j\in S_v} R_{v,j}.
\end{equation}
The simple subsampling estimator is recovered by taking \(c_v\equiv0\). In this case, \(R_{v,j}=r_v(z_{v,j})\), and \eqref{eq:Hv-hat-special} reduces to
\begin{equation*}
\label{eq:Hv-hat-simple}
\widehat H_v(Z_v,S_v)
=
\frac{n_v}{m_v}\sum_{j\in S_v}r_v(z_{v,j}).
\end{equation*}
{\color{purple}For Algorithms~\ref{alg:noisy-gibbs} and~\ref{alg:fresh-field-gibbs}, take \(\xi_v=S_v\) and let \(Q_v\) be the uniform distribution on subsets of \(\{1,\ldots,n_v\}\) of size \(m_v\).}

For this estimator, the field error is
\[
e_v(S_v)
=
\widehat H_v(Z_v,S_v)-H_v(Z_v)
=
n_v\left(
\frac1{m_v}\sum_{j\in S_v}R_{v,j}
-
\overline R_v
\right).
\]
Thus \(\mathbb E_{S_v}[e_v(S_v)]=0\). Set
\[
b_v
=
\max_{1\le j\le n_v}
\big|R_{v,j}-\overline R_v\big|.
\]
The standard exponential-moment bound for sampling without replacement gives
\begin{equation}\label{eq:finite-pop-mgf}
\E_{S_v}\big[\exp(t e_v(S_v))\big]
\le
\exp\left(
\frac{t^2n_v^2b_v^2}{2m_v}
\right),
\qquad t\in\mathbb R;
\end{equation}
see, for example, \cite[Lemmas~1.1 and~1.3]{bardenet_maillard_2015}.

\begin{corollary}
\label{cor:random-field-budget}
Suppose that \(G\) has maximum degree at most \(D\) and that \(q<1\). Then \purple{there is a constant \(C_{D,q}<\infty\), independent of \(|V|\), \(Z\), and the field estimators, such that}
\begin{equation}\label{eq:subsampling-random-field-bound}
\purple{\max\!\left\{
\dH^2\bigl(\widetilde\mu(\cdot\mid Z),\mu(\cdot\mid Z)\bigr),
\dH^2\bigl(\overline\mu(\cdot\mid Z),\mu(\cdot\mid Z)\bigr)
\right\}}
\le
C_{D,q}
\sum_{v\in V}
\frac{n_v^4b_v^4}{m_v^2}.
\end{equation}
\end{corollary}

\begin{proof}
Equation~\eqref{eq:finite-pop-mgf} verifies Assumption~\ref{ass:e_v} with
\[
\vartheta_v^2=\frac{n_v^2b_v^2}{m_v}.
\]
\purple{Substitution into Theorems~\ref{thm:noisy-gibbs-cv} and~\ref{thm:fresh-field-gibbs} gives \eqref{eq:subsampling-random-field-bound}.}
\end{proof}

\subsubsection{\purple{Statistical consequences}} \label{SubsecInterp}

{\color{purple}It is not immediately clear how Corollary~\ref{cor:random-field-budget} can be used to tune either algorithm. We give a simple comparison of the sufficient subsample sizes obtained from our bounds and from classical perturbation analysis.}

{\color{purple}
Write \(N=|V|\) and \(\vartheta_v=n_vb_v/\sqrt{m_v}\). Let \(K^\star\) denote either \(K_{\widetilde H}\), the usual random-scan Gibbs kernel with invariant distribution \(\widetilde\mu\), or the kernel \(\overline K\) from Algorithm~\ref{alg:fresh-field-gibbs}, and let \(\mu^\star\) be the corresponding invariant distribution. For Algorithm~\ref{alg:noisy-gibbs}, \eqref{eq:effective-ising-field} and Assumption~\ref{ass:e_v} give
\[
\big|\widetilde H_v(Z_v)-H_v(Z_v)\big|
\le
\frac{\vartheta_v^2}{4}.
\]
For Algorithm~\ref{alg:fresh-field-gibbs}, Taylor's theorem, \(\E[e_v]=0\), and \(\E[e_v^2]\le\vartheta_v^2\) give, uniformly in the boundary field \(a\),
\[
\left|\E[\psi(a+e_v)]-\psi(a)\right|
\le
\frac12\|\psi''\|_\infty\vartheta_v^2.
\]
Let \(W_{\mathrm{Ham}}\) be the Wasserstein distance for Hamming distance. Coupling the selected vertex and then the two Bernoulli draws gives
\begin{equation}\label{eq:image-one-step-wasserstein}
\varepsilon
:=
\sup_{\sigma\in\Omega}
W_{\mathrm{Ham}}\bigl(K^\star(\sigma,\cdot),K(\sigma,\cdot)\bigr)
\lesssim
\frac1N\sum_{v\in V}\vartheta_v^2.
\end{equation}
Under \(q<1\), the standard Dobrushin coupling contracts Hamming distance by the factor \(\gamma:=1-(1-q)/N\) per update; see, for example, \cite[Chapter~15]{Levin2008MarkovCA}. Indeed, for configurations differing only at \(u\), coupling the selected vertex gives expected Hamming distance at most \(1-N^{-1}+N^{-1}\sum_{v:\,\{u,v\}\in E}\tanh(\beta J_{uv})\le\gamma\), and the general statement follows by path coupling. Therefore
\[
W_{\mathrm{Ham}}(\lambda K,\lambda'K)
\le
\gamma W_{\mathrm{Ham}}(\lambda,\lambda')
\]
for all probability measures \(\lambda,\lambda'\). Using stationarity and \eqref{eq:image-one-step-wasserstein},
\begin{align*}
W_{\mathrm{Ham}}(\mu^\star,\mu)
&\le
W_{\mathrm{Ham}}(\mu^\star K^\star,\mu^\star K)
+
W_{\mathrm{Ham}}(\mu^\star K,\mu K)\\
&\le
\varepsilon+\gamma W_{\mathrm{Ham}}(\mu^\star,\mu).
\end{align*}
Since total variation is bounded by \(W_{\mathrm{Ham}}\), it follows that
\[
\dTV(\mu^\star(\cdot\mid Z),\mu(\cdot\mid Z))
\le
\frac{\varepsilon}{1-\gamma}
\lesssim
\sum_{v\in V}\vartheta_v^2.
\]
}
Thus the classical argument requires
\[
\sum_{v\in V}\vartheta_v^2\longrightarrow0.
\]
By contrast, Corollary~\ref{cor:random-field-budget} only requires
\[
\sum_{v\in V}\vartheta_v^4\longrightarrow0.
\]
Thus, in the homogeneous case \(\vartheta_v\le \vartheta_N\), the sufficient rate is improved from \(\vartheta_N=o(N^{-1/2})\) to \(\vartheta_N=o(N^{-1/4})\).

The bound in \eqref{eq:subsampling-random-field-bound} is useful when the residuals in the subsampling estimator are sufficiently small. One way to make \(b_v\) small is to use a local Taylor control variate. Suppose, for illustration, that the observations take values in a bounded interval and that \(r_v\in C^2\) with \(\sup_z |r_v''(z)|\le L_v\). Partition this interval into subintervals \(\{I_{v,k}:1\le k\le B_v\}\), choose \(x_{v,k}\in I_{v,k}\), and set
\[
    c_v(z)
    =
    r_v(x_{v,k})+r_v'(x_{v,k})(z-x_{v,k}),
    \qquad z\in I_{v,k}.
\]
If \(\Delta_v=\max_k |I_{v,k}|\), Taylor's theorem gives
\[
    |r_v(z)-c_v(z)|\lesssim L_v \Delta_v^2,
\]
and hence
\[
    b_v\lesssim L_v \Delta_v^2.
\]
In particular, if the intervals have comparable lengths \(\Delta_v\asymp B_v^{-1}\), we have
\(b_v^2\lesssim L_v^2B_v^{-4}\).  In the homogeneous case
\(n_v=n\), \(m_v=m\), \(B_v=B\), and \(L_v\le L_0\), Corollary~\ref{cor:random-field-budget} gives
\[
    \dH^2\bigl(\widetilde\mu(\cdot\mid Z),\mu(\cdot\mid Z)\bigr)
    \lesssim
    \frac{Nn^4}{m^2}B^{-8}.
\]
The classical condition \(\sum_{v\in V}\vartheta_v^2\to0\) requires
\(m\gg Nn^2B^{-4}\), whereas Corollary~\ref{cor:random-field-budget} gives the sufficient condition
\begin{equation*}
m \gg \sqrt{N}\,n^2B^{-4}.
\end{equation*}

{\color{purple}Theorems~\ref{thm:noisy-gibbs-cv} and~\ref{thm:fresh-field-gibbs} both give a bound proportional to \(\sum_{v\in V}\vartheta_v^4\), and Corollary~\ref{cor:random-field-budget} gives the corresponding control-variate bound. Section~\ref{subsec:Example3} obtains the corresponding matching bound by direct calculation.}

\subsection{Approximate Weight Evaluation for Bayesian Matchings}\label{subsec:Example3}

{\color{purple}We next consider single-edge Gibbs samplers for weighted matchings. We first define the exact and noisy chains, then relate them to record linkage, and finally state the perturbation bound and its consequences.}

\subsubsection{\purple{Weighted matching chains}}

Fix a bipartite graph $\cG=(\cU\sqcup \cW,\cE)$ of candidate matchings. A state of the Markov chain is a \textit{matching}---that is, a subset \(M\subseteq\cE\) of the edges. For any $M\subseteq \cE$, define
\[
    x_e(M)=\mathbf 1\{e\in M\}, \qquad e\in \cE,
\]
and, for \(A\subseteq \cE \), write
\[
x_A(M)=(x_e(M))_{e\in A}.
\]
Then the state space can also be written as
\[
    \Omega_\cG
    =
    \left\{
    M\subseteq \cE:
    x_e(M)+x_f(M)\le 1 \text{ whenever } e\sim f
    \right\},
\]
where \(e\sim f\) denotes distinct edges sharing an endpoint. Define a collection of \purple{real} edge weights \(\{w_e\}_{e \in \cE}\). We consider the target distribution
\begin{equation}\label{eq:matching-posterior}
    \mu_w(M)
    =
    \frac{1}{\mathcal Z_{\cG,w}}
    \exp\left\{\sum_{e\in M} w_e\right\}\footnote{Vertex weights for unmatched records can be absorbed into the edge
weights: if edges have activities \(\lambda_{uv}>0\) and unmatched vertices have weights \(\theta_r>0\), then, writing \(V(M)\) for the vertices covered by \(M\), \(\prod_{(u,v)\in M}\lambda_{uv}\prod_{r\notin V(M)}\theta_r\) is equal, up to
the \(M\)-independent factor \(\prod_{r\in\cU\sqcup\cW}\theta_r\), to
\(\prod_{(u,v)\in M}\lambda_{uv}/(\theta_u\theta_v)\).  Thus the corresponding edge weight is
\(w_{uv}=\log\lambda_{uv}-\log\theta_u-\log\theta_v\).},
    \qquad M\in \Omega_\cG,
\end{equation}
where  \(\mathcal Z_{\cG,w}\) is the normalizing constant.

Let $\chi(t)=e^t/(1+e^t)$.
For $e\in\cE$, let
\[
    N(e)=\{f\in\cE:f\sim e\}
\]
denote the set of edges adjacent to $e$, and set
\[
    I_e(M)=\mathbf 1\{M\cap N(e)=\emptyset\}.
\]
Thus $I_e(M)=1$ means that no edge adjacent to \(e\) is currently selected in $M$, so that the edge $e$ is available conditional on all other edge variables.

We denote by \(K_w\) the standard single-edge Gibbs kernel described in Algorithm~\ref{alg:matching-exact}.

\begin{algorithm}[H]
\caption{One step of $K_w$}
\label{alg:matching-exact}
\begin{algorithmic}[1]
\Require Current matching $M\in\Omega_\cG$, edge weights $w=(w_e)_{e\in\cE}$.
\State Sample $e\sim \mathrm{Unif}(\cE)$.
\State Sample $B\sim\mathrm{Bernoulli}( I_e(M) \, \chi(w_e))$.
\State Set
\[
    M'=
    \begin{cases}
    M\cup\{e\}, & B=1,\\
    M\setminus\{e\}, & B=0.
    \end{cases}
\]
\State \Return $M'$.
\end{algorithmic}
\end{algorithm}

Let \(Y\sim\mu_w\), set \(X_e=x_e(Y)\) for \(e\in\cE\), and write
\(X_A=(X_e)_{e\in A}\) for \(A\subseteq\cE\). Then, for every \(e\in\cE\) and every \(M\in\Omega_\cG\),
\begin{equation*}\label{eq:matching-conditional}
    \mu_w\!\left(
        X_e=1\mid X_{\cE\setminus\{e\}}=x_{\cE\setminus\{e\}}(M)
    \right)
    =
    I_e(M)\chi(w_e).
\end{equation*}

{\color{purple}We next define a noisy version of Algorithm~\ref{alg:matching-exact}. We begin with the generic algorithm and relate it to record linkage afterward.} For each \(e\in\cE\), let \(Q_e\) be a probability distribution on \(\mathbb R\). Whenever \(e\) is selected and available, draw a new variable \(\xi_e\sim Q_e\), independently of the current matching and all previous draws, and replace \(w_e\) by
\begin{equation} \label{EqNoisyLinkage}
    \widehat w_e=w_e+\xi_e.
\end{equation}
This leads to Algorithm~\ref{alg:matching-noisy}. Let \(\widetilde K\) be the associated kernel.

\begin{algorithm}[H]
\caption{One step of $\widetilde K$}
\label{alg:matching-noisy}
\begin{algorithmic}[1]
\Require Current matching $M\in\Omega_\cG$, edge weights $w=(w_e)_{e\in\cE}$.
\State Sample $e\sim \mathrm{Unif}(\cE)$.
\State If $I_e(M)=1$, draw an independent $\xi_e\sim Q_e$; otherwise set $\xi_e=0$.
\State Sample $B\sim\mathrm{Bernoulli}( I_e(M) \, \chi(w_e+\xi_e))$.
\State Set
\[
    M'=
    \begin{cases}
    M\cup\{e\}, & B=1,\\
    M\setminus\{e\}, & B=0.
    \end{cases}
\]
\State \Return $M'$.
\end{algorithmic}
\end{algorithm}

We continue with the formal exposition. For each \(e\in\cE\), define
\begin{equation}\label{eq:matching-tilde-weights}
    \widetilde p_e
    =
    \mathbb E\bigl[\chi(w_e+\xi_e)\bigr],
    \qquad
    \widetilde w_e
    =
    \logit(\widetilde p_e)
    =
    \log\left(\frac{\widetilde p_e}{1-\widetilde p_e}\right).
\end{equation}
\purple{Because \(0<\chi(t)<1\) for every \(t\in\mathbb R\), we have \(0<\widetilde p_e<1\), so \(\widetilde w_e\) is finite.}
Then, for every matching \(M\),
\begin{equation}\label{eq:noisy-matching-kernel-identity}
    \widetilde K(M,\cdot)=K_{\widetilde w}(M,\cdot).
\end{equation}
Indeed, both kernels set \(X_e=0\) when \(I_e(M)=0\). When
\(I_e(M)=1\), both set \(X_e=1\) with probability
\[
    \mathbb E\bigl[\chi(w_e+\xi_e)\bigr]
    =
    \widetilde p_e
    =
    \chi(\widetilde w_e).
\]
Thus \(\widetilde K\) is the exact single-edge Gibbs kernel with edge weights
\(\widetilde w=(\widetilde w_e)_{e\in\cE}\). Its invariant distribution is
\begin{equation}\label{eq:matching-modified-posterior}
    \mu_{\widetilde w}(M)
    =
    \frac{1}{\mathcal Z_{\cG,\widetilde w}}
    \exp\left\{\sum_{e\in M}\widetilde w_e\right\},
    \qquad M\in\Omega_\cG.
\end{equation}

\subsubsection{\purple{Bayesian record linkage}}

{\color{purple}We now explain how this class arises in a simple Bayesian record-linkage problem.} Since this problem is less common and less simple than the image reconstruction problem, we give a slightly more detailed informal description of the statistical context and our results first.

In the simplest version of the record-linkage problem, we have two different sources of records (e.g., the census and Facebook accounts) and we assume every individual occurs at most once in each source (e.g., a person should have at most one Facebook account and should appear exactly once in the census---although both assumptions can be violated). The goal is to link the two sources of records by matching records associated with an individual in source 1 to the same individual in source 2. To set notation, we denote by $\cU, \cW$ the two sets of records. When both sources are large (e.g., in the tens or hundreds of millions), it is computationally intractable to carefully compare every candidate match, and statistical models are usually not built on the set of all possible matchings. Instead, simple preprocessing rules called ``blocking'' rules extract a much smaller list of candidate pairs which are then compared using a statistical model.\footnote{Since this may be a source of confusion for statisticians who have not encountered blocking rules before, we acknowledge that it is not very obvious that one can actually construct a useful list of candidate matches without looking at all pairs of records. However, there are well-established techniques for doing so. The simplest is to consider a function $f:\cU\sqcup\cW\to S$ for some large set $S$, and to include any pair $u,w$ for which $f(u) = f(w)$. These candidate lists can be constructed explicitly in time at most $O(|\cU\sqcup\cW| + \sum_{s \in S} |f^{-1}(s)|^{2})$, which can be quite small if $S$ is quite large and $f$ spreads out points. See, e.g., \cite{quinn2023exploring} for practical details using a popular and widely applicable class of functions $f$.} Formally, we denote by $\cG=(\cU\sqcup \cW,\cE)$
the bipartite graph of candidate matchings, where each edge \(e=(u,w)\in\cE\) represents a match between \(u\in\cU\) and \(w\in\cW\) that is not ruled out by the blocking procedure.

For each candidate edge \(e=(u,w)\), the statistical model assigns an edge weight \(w_e\) measuring how compatible the two records are as a potential match. This weight may depend on a large amount of information attached to the records, such as names, addresses, dates, text fields, or other features. The exact Gibbs update uses the full edge weight for the selected edge. In large record-linkage problems, repeatedly computing these weights can be expensive. A noisy update instead uses a randomized estimator \(\widehat w_e\) of \(w_e\) and performs the Gibbs update with this estimate.

{\color{purple}The target in \eqref{eq:matching-posterior} does not by itself describe data, and our bounds do not require a statistical interpretation.} However, a large class of simple Bayesian record-linkage models can be written in this form. We give a simple example from the literature. Assume that there is a ``true" matching of the vertices, and that the data used to compare records $u,v$ is sampled independently from some model, conditional on this true matching. We can then take $w_{u,v}$ to be the log-likelihood-ratio weight comparing the hypotheses that $(u,v)$ is and is not a true match; see Equation~(1) of \cite{Sadinle2017Bayesian} and the surrounding discussion. If the prior is uniform over all matchings allowed by the graph $\cG$, then the nonmatch terms cancel and Equation \eqref{eq:matching-posterior} is the corresponding posterior distribution.

{\color{purple} When defining our generic Markov chain, we did not specify how to construct the random variable \(\widehat w_e\) in Equation~\eqref{EqNoisyLinkage}, we merely asserted a list of assumptions about its properties. In the context of noisy MCMC, we informally think of \(\widehat w_e\) as a randomized estimator of \(w_e\), with \(\xi_e=\widehat w_e-w_e\). This randomized estimator can come from \textit{e.g.} subsampling pieces of evidence associated with the records $(u,v) = e$. }

\subsubsection{Main Bound}

 We make the following assumptions.

\begin{assumption}\label{ass:degree-4.3}
There exists a constant \(0<D<\infty\) such that
\[
    \Delta(\cG)\le D.
\]
\end{assumption}

\begin{assumption}\label{ass:edge-weight}
There exists a constant \purple{\(0\le W<\infty\)} such that the edge weights in \eqref{eq:matching-posterior} satisfy
\[
    |w_e|\le W
    \qquad
    \text{for all } e\in\cE.
\]
\end{assumption}

Assumption \ref{ass:edge-weight} is typically \textit{not} satisfied for classical record-linkage models such as Fellegi-Sunter in the regime that one accumulates more and more data about each individual, because those models treat each piece of evidence as independent. In practice, record-linkage methods will typically include ways to ``clip" extremely high-certainty match probabilities to avoid numerical issues (see \textit{e.g.} the popular package \cite{linacre2022splink}).

Assumption~\ref{ass:edge-weight} ensures that the conditional inclusion probabilities \(\chi(w_e)\) when the edge is available are uniformly bounded away from \(0\) and \(1\); this follows immediately from the uniform bound on \(|w_e|\).

\begin{assumption}\label{ass:noise-4.3}
The noise variables in \eqref{EqNoisyLinkage} satisfy
\[
    \mathbb E[\xi_e]=0,
    \qquad
    s_e^2=\mathbb E[\xi_e^2]<\infty,
    \qquad e\in\cE.
\]
\end{assumption}

\begin{theorem}\label{thm:noisy-matching}
Let \(\cG_N=(\cU_N\sqcup\cW_N,\cE_N)\) be a sequence of bipartite candidate graphs
satisfying Assumptions~\ref{ass:degree-4.3}--\ref{ass:noise-4.3}, \purple{with the same constants \(D\) and \(W\) for every \(N\)}. Then there exist constants
\(c_W>0\) and \(C_{D,W}<\infty\) such that, if
\[
    \max_{e\in\cE_N}s_{N,e}^2\le c_W,
\]
then
\begin{equation}\label{eq:noisy-matching-main-bound}
    \dH^2(\mu_{\widetilde w_N},\mu_{w_N})
    \le
    C_{D,W}\sum_{e\in\cE_N}s_{N,e}^4.
\end{equation}
Consequently, if \(N_{\cE}=|\cE_N|\) and \(s_{N,e}\le s_N\) for all \(e\in\cE_N\), then
\[
    \dH(\mu_{\widetilde w_N},\mu_{w_N})\to0
    \qquad
    \text{whenever}
    \qquad
    s_N=o(N_{\cE}^{-1/4}).
\]
\end{theorem}

\begin{proof}

\purple{The proof is given in Appendix~\ref{app:matching-proof}; it combines Lemma~\ref{lem:random-weight-logit} with Corollary~\ref{cor:matching-perturbation}.}
\end{proof}

\subsubsection{\purple{Statistical consequences}}

\begin{remark}
Recall that \(s_e^2=\E[\xi_e^2]\). Equation~\eqref{eq:random-weight-logit} gives \( |\widetilde w_e-w_e|\le C_Ws_e^2\). Substituting this bound into the
matching perturbation bound \eqref{eq:matching-perturbation} gives
\eqref{eq:noisy-matching-main-bound}.

Taylor's theorem gives
\[
    \big|\widetilde p_e-\chi(w_e)\big|
    \le
    \frac12\|\chi''\|_\infty s_e^2.
\]
Averaging over the uniformly selected edge therefore gives
\begin{equation}\label{eq:noisy-matching-one-step}
    \sup_{M\in\Omega_\cG}
    \dTV\bigl(\widetilde K(M,\cdot),K_w(M,\cdot)\bigr)
    \le
    \frac{\|\chi''\|_\infty}{2|\cE|}
    \sum_{e\in\cE}s_e^2.
\end{equation}
Thus the direct estimate \eqref{eq:noisy-matching-one-step} involves
\(\sum_{e\in\cE}s_e^2\). By contrast,
\eqref{eq:noisy-matching-kernel-identity} identifies \(\widetilde K\) with an exact
Gibbs kernel, and \eqref{eq:random-weight-logit} together with
\eqref{eq:matching-perturbation} yields the sum \(\sum_{e\in\cE}s_e^4\).
\end{remark}

In record-linkage terms, \(s_e\) is the standard deviation of the error \(\widehat w_e-w_e\). Theorem~\ref{thm:noisy-matching} therefore specifies the required accuracy of a randomized estimator of \(w_e\). To avoid too much repetition, we don't discuss how to turn this bound into a bound on the computational cost of pre-computing a concrete family of control variates. See the end of Section \ref{SubsecInterp} for an example of how to do exactly this step in the context of another (very similar) noisy MCMC chain.

Appendix~\ref{app:comparison-previous} proves Proposition~\ref{prop:hellinger-block-bound} and Lemma~\ref{LemmaPreDoesntWork}, and Appendix~\ref{app:matching-proof} proves Theorem~\ref{thm:noisy-matching}.

\clearpage

\normalem
\bibliographystyle{plain}
\bibliography{ref}

\clearpage

\appendix

\section{Comparison with Earlier Bounds}\label{app:comparison-previous}

We first prove Proposition~\ref{prop:hellinger-block-bound}. Similarly to the local entropy as a function of the boundary configuration, for $A\subseteq V$ and $h:\Omega\to\mathbb R$, we write
\[
\bigl(\Var_A h\bigr)(\tau_{A^c})
=
\Var_{\mu_A^{\tau_{A^c}}}(h).
\]

\begin{proof}[Proof of Proposition~\ref{prop:hellinger-block-bound}]
Applying \eqref{Def:LocalEntropy} to \(1+th\) for sufficiently small \(t\), and using the second-order expansion
\[
\Ent_A^{\tau_{A^c}}(1+th)
=
\frac{t^2}{2}\Var_{\mu_A^{\tau_{A^c}}}(h)+o(t^2),
\qquad t\to 0.
\]
Substituting this expansion into \eqref{Ineq-BlockFactor}, dividing both sides by \(t^2/2\), and letting \(t\to0\) gives
\begin{equation}\label{eq:variance-block-factorization}
\gamma(\alpha)\Var_\mu(h)
\leq
C_{\mathrm{BF}}
\sum_{A\in\mathcal A}
\alpha_A\,\mu(\Var_A h).
\end{equation}
Set $f=\nu/\mu$ and $u=\sqrt f$. Since $\mu(f)=1$,
\begin{align}
\dH^2(\nu,\mu) &= \mu ((u-1)^2) = 2\bigl(1-\mu(u)\bigr)\notag\\
&\leq
2\bigl(1-\mu(u)^2\bigr)
=
2\Var_\mu(u).
\label{eq:hellinger-global-variance}
\end{align}

Fix $A\in\mathcal A$ and a boundary condition $\tau_{A^c}$. Set
\[
m=\mu_A^{\tau_{A^c}}(f),
\qquad
b=\mu_A^{\tau_{A^c}}(u).
\]
By Cauchy--Schwarz, $0\leq b\leq\sqrt m$. If $m=0$, the conditional variance below is zero. If $m>0$, the density of $\nu_A^{\tau_{A^c}}$ relative to $\mu_A^{\tau_{A^c}}$ is $f/m$.
Indeed, we have
\begin{align*}
m &= \sum_{\eta_A} \mu_A^{\tau_{A^c}}(\eta_A) f(\eta_A \tau_{A^c}) = \frac{1}{\mu_{A^c}(\tau_{A^c})}\sum_{\eta_A} \nu(\eta_A \tau_{A^c}) = \frac{\nu_{A^c}(\tau_{A^c})}{\mu_{A^c}(\tau_{A^c})}
\end{align*}
and
\begin{align*}
\nu_{A}^{\tau_{A^c}}(\eta_A) &=\frac{f(\eta_A\tau_{A^c})\mu(\eta_A \tau_{A^c})}{\nu_{A^c}(\tau_{A^c})}\\
& = f(\eta_A\tau_{A^c})\frac{\mu(\eta_A \tau_{A^c})}{\mu_{A^c}(\tau_{A^c})}\, \frac{\mu_{A^c}(\tau_{A^c})}{\nu_{A^c}(\tau_{A^c})} = \frac{f(\eta_A\tau_{A^c})}{m} \mu_A^{\tau_{A^c}}(\eta_A).
\end{align*}
Hence,
\begin{align*}
\dH^2\left(\nu_A^{\tau_{A^c}}, \mu_A^{\tau_{A^c}} \right)&= \mu_A^{\tau_{A^c}} \left[ \left(\sqrt{\frac{f}{m}}-1\right)^2 \right]\\
& =\mu_A^{\tau_{A^c}} \left(\frac{f}{m}-2\frac{\sqrt{f}}{\sqrt{m}}+1\right) = \frac{2}{\sqrt{m}} (\sqrt{m}-b)
\end{align*}
and
\begin{align*}
\Var_{\mu_A^{\tau_{A^c}}}(u)
&=
m-b^2=(\sqrt m-b)(\sqrt m+b)\\
&\leq
2\sqrt m(\sqrt m-b)=
m\,\dH^2\left(
\nu_A^{\tau_{A^c}},
\mu_A^{\tau_{A^c}}
\right).
\end{align*}
Integrating with respect to $\mu_{A^c}$ and using $m\,d\mu_{A^c}=d\nu_{A^c}$ gives
\begin{equation}\label{eq:conditional-variance-hellinger}
\mu(\Var_A u)
\leq
\E_{\nu_{A^c}}
\left[
\dH^2\left(
\nu_A^{\tau_{A^c}},
\mu_A^{\tau_{A^c}}
\right)
\right].
\end{equation}
Combining \eqref{eq:hellinger-global-variance}, \eqref{eq:variance-block-factorization}, and \eqref{eq:conditional-variance-hellinger} proves Proposition~\ref{prop:hellinger-block-bound}.
\end{proof}

\subsection{Sequential Factorizations and Treewidth}

This subsection proves the treewidth estimate used in Lemma~\ref{LemmaPreDoesntWork}. We begin by quickly summarizing two standard facts and some standard notation from combinatorics. 

For disjoint sets $A,B,C\subseteq V$, say that $C$ \emph{separates} $A$ from $B$ in $G$ if every path in $G$ from $A$ to $B$ contains a vertex of $C$. A consequence of \cite[Corollary~7.7 and Theorem~6.1]{FallatLauritzenSadeghiUhlerWermuthZwiernik2017} is that, for a fixed graph $G = (V,E)$ and an Ising measure (that is, measure of the form \eqref{EqDefIsingGibbs}) with $\beta J_{uv}>0$ on every edge $(u,v) \in E$, 
\[
    X_A\perp X_B\mid X_C
    \quad\Longrightarrow\quad
    C \text{ separates } A \text{ from } B \text{ in } G.
\]

We also use the definition of the \textit{treewidth} of a graph $G$ in terms of ordered eliminations. Fix a graph $G = (V,E)$. Given an ordering of the vertices, eliminate them one at a time, adding an edge between every pair of remaining neighbors of a vertex immediately before that vertex is removed. The \textit{width} of the ordering is the largest number of remaining neighbors of a vertex when it is removed. The minimum width over all orderings is $\operatorname{tw}(G)$; see \cite[Theorem~36]{bodlaender1998partial} for a proof that this gives the same result as the original definition of treewidth.

\begin{lemma}\label{lem:factorization-treewidth}
Let $\mu$ be a ferromagnetic Ising measure on a finite graph $G=(V,E)$, with interaction coefficients $\beta J_{uv}>0$ for every $\{u,v\}\in E$. Suppose that a partition $\{S_j\}_{j=1}^{\ell}$ and sets $\{\Pi_j\}_{j=1}^{\ell}$ satisfy the $\mu$-factorization in \eqref{Bayesian-Factor} and \eqref{EqContainmentFactor}. Then
\[
    \operatorname{tw}(G)
    \leq
    \max_{1\leq j\leq\ell}|S_j\cup\Pi_j|-1.
\]
\end{lemma}

\begin{proof}
Let $X\sim\mu$, and set
\[
    P_j=\bigcup_{i=1}^j S_i.
\]
Summing the factorization \eqref{Bayesian-Factor} over the blocks $S_\ell,S_{\ell-1},\ldots,S_{j+1}$, in that order, gives
\[
    \mu(\sigma_{P_j})
    =
    \prod_{i=1}^j \mu(\sigma_{S_i}\mid\sigma_{\Pi_i}).
\]
The same identity with $j-1$ in place of $j$ therefore implies
\[
    \mu(\sigma_{S_j}\mid\sigma_{P_{j-1}})
    =
    \mu(\sigma_{S_j}\mid\sigma_{\Pi_j}).
\]
Thus
\[
    X_{S_j}\perp X_{P_{j-1}\setminus\Pi_j}\mid X_{\Pi_j}.
\]
The implication above shows that $\Pi_j$ separates $S_j$ from $P_{j-1}\setminus\Pi_j$ in $G$.

Now eliminate the vertices block by block in the order
\[
    S_\ell,S_{\ell-1},\ldots,S_1,
\]
using any order within each block. We use the elementary observation that every edge in the current graph between two remaining vertices corresponds to a path between those vertices in the original graph whose internal vertices have already been eliminated. This follows by induction over the elimination steps.

At the start of block $S_j$, the vertices in $V\setminus P_j$ have already been eliminated. If some $x\in S_j$ were adjacent in the current graph to a vertex $y\in P_{j-1}\setminus\Pi_j$, the observation above would give a path from $x$ to $y$ in $G$ whose internal vertices lie in $V\setminus P_j$. This path avoids $\Pi_j$, contradicting the separation above. Hence every current neighbor of a vertex in $S_j$ lies in $S_j\cup\Pi_j$.

This remains true as the vertices of $S_j$ are eliminated, since the new edges added during the block join only vertices already in $S_j\cup\Pi_j$. Consequently, every vertex in the resulting elimination ordering has at most
\[
    \max_{1\leq j\leq\ell}|S_j\cup\Pi_j|-1
\]
remaining neighbors when it is removed. The elimination characterization of treewidth completes the proof.
\end{proof}

\section{Perturbation Bounds for Weighted Matchings}\label{app:matching-proof}

We prove Theorem~\ref{thm:noisy-matching}. The argument applies to arbitrary finite
graphs; bipartiteness is not used.

\subsection{Monomer susceptibilities}

Let \(H=(V(H),E(H))\) be a finite graph, and write \(\Delta(H)\) for its
maximum vertex degree. For vertices \(a,b\in V(H)\), write \(a\sim b\) when
\(\{a,b\}\in E(H)\). For \(a\in V(H)\), let \(H-a\) denote the induced graph
obtained by deleting \(a\) and its incident edges. For \(e\in E(H)\), let \(H-e\)
denote the graph obtained by deleting \(e\). Repeated deletions are interpreted
successively, and edge weights are restricted to the surviving edges without further
notation.

Let \(u=(u_f)_{f\in E(H)}\) be a collection of edge weights, with activities
\(\lambda_f=e^{u_f}\). Let \(\Omega_H\) denote the set of matchings of \(H\). For
\(M\in\Omega_H\), set \(x_f(M)=\mathbf 1\{f\in M\}\), and define
\begin{equation}\label{eq:matching-partition-function}
    \mathcal Z_{H,u}
    =
    \sum_{M\in\Omega_H}
    \exp\left\{\sum_{f\in M}u_f\right\}.
\end{equation}
Let \(\mu_{H,u}\) be the corresponding matching measure. For
\(M\sim\mu_{H,u}\), write \(X_f=x_f(M)\). For \(a\in V(H)\), define
\[
    p_a(H,u)=\frac{\mathcal Z_{H-a,u}}{\mathcal Z_{H,u}}.
\]
This is the probability that \(a\) is unmatched under \(\mu_{H,u}\). When \(u\) is
fixed, write \(p_a(H)=p_a(H,u)\). Partitioning matchings according to whether \(a\)
is unmatched or matched to a neighbor gives
\begin{equation}\label{eq:matching-vertex-deletion}
    \mathcal Z_{H,u}
    =
    \mathcal Z_{H-a,u}
    +
    \sum_{b:b\sim a}\lambda_{ab}\mathcal Z_{H-a-b,u}.
\end{equation}
Dividing \eqref{eq:matching-vertex-deletion} by \(\mathcal Z_{H-a,u}\) and taking reciprocals gives
\begin{equation}\label{eq:monomer-recursion-direct}
    p_a(H)
    =
    \frac{1}{1+\sum_{b:b\sim a}\lambda_{ab}p_b(H-a)}.
\end{equation}

\begin{lemma}\label{lem:monomer-susceptibility}
Assume that \(\Delta(H)\le D\) and \(0<\lambda_f\le\Lambda\) for every
\(f\in E(H)\). For \(a\in V(H)\), set
\[
    S_a(H,u)
    =
    \sum_{f\in E(H)}
    \left|
        \frac{\partial}{\partial u_f}\log p_a(H,u)
    \right|.
\]
Then
\[
    S_a(H,u)\le D\Lambda.
\]
\end{lemma}

\begin{proof}
Induct on \(|V(H)|\). If \(a\) is isolated, then \(S_a(H,u)=0\). Otherwise,
for \(b\sim a\), set
\[
    \alpha_{ab}
    =
    \frac{\lambda_{ab}p_b(H-a)}
    {1+\sum_{c:c\sim a}\lambda_{ac}p_c(H-a)}.
\]
Since \(p_b(H-a)\le1\),
\begin{equation}\label{eq:alpha-sum-bound}
    \sum_{b:b\sim a}\alpha_{ab}
    \le
    \frac{D\Lambda}{1+D\Lambda}.
\end{equation}
Differentiate \eqref{eq:monomer-recursion-direct}:
\[
    \frac{\partial}{\partial u_f}\log p_a(H)
    =
    -\sum_{b:b\sim a}\alpha_{ab}
    \left(
        \mathbf 1\{f=\{a,b\}\}
        +
        \frac{\partial}{\partial u_f}\log p_b(H-a)
    \right),
\]
where the last derivative is zero when \(f\notin E(H-a)\). Sum over \(f\),
apply the triangle inequality, and use the induction hypothesis:
\[
    S_a(H,u)
    \le
    \sum_{b:b\sim a}\alpha_{ab}\{1+S_b(H-a,u)\}
    \le
    (1+D\Lambda)\sum_{b:b\sim a}\alpha_{ab}.
\]
The result follows from \eqref{eq:alpha-sum-bound}.
\end{proof}

\subsection{Covariance row sums}

\begin{lemma}\label{lem:matching-covariance-row}
Assume that \(\Delta(H)\le D\) and \(0<\lambda_f\le\Lambda\) for every
\(f\in E(H)\). Then, for every \(e\in E(H)\),
\begin{equation}\label{eq:matching-covariance-row-sum}
    \sum_{f\in E(H)}
    \left|\Cov_{\mu_{H,u}}(X_e,X_f)\right|
    \le
    C_{\mathrm{cov}}(D,\Lambda),
\end{equation}
where
\[
    C_{\mathrm{cov}}(D,\Lambda)=\frac14(1+2D\Lambda).
\]
Consequently, for every \(\theta=(\theta_f)_{f\in E(H)}\),
\begin{equation}\label{eq:matching-linear-variance}
    \Var_{\mu_{H,u}}\left(\sum_{f\in E(H)}\theta_fX_f\right)
    \le
    C_{\mathrm{cov}}(D,\Lambda)
    \sum_{f\in E(H)}\theta_f^2.
\end{equation}
\end{lemma}

\begin{proof}
Fix \(e=\{a,b\}\in E(H)\), and define
\[
    R_e(H,u)
    =
    \frac{\mu_{H,u}(X_e=1)}{\mu_{H,u}(X_e=0)}.
\]
Partitioning matchings according to \(X_e\) gives the first equality in
\eqref{eq:edge-odds-direct}; the second follows from the definition of \(p_a\):
\begin{equation}\label{eq:edge-odds-direct}
    R_e(H,u)
    =
    e^{u_e}\frac{\mathcal Z_{H-a-b,u}}{\mathcal Z_{H-e,u}}
    =
    e^{u_e}p_a(H-e,u)p_b(H-e-a,u).
\end{equation}
Applying Lemma~\ref{lem:monomer-susceptibility} to \eqref{eq:edge-odds-direct} gives
\begin{equation}\label{eq:matching-log-odds-gradient}
    \sum_{f\in E(H)}
    \left|
        \frac{\partial}{\partial u_f}\log R_e(H,u)
    \right|
    \le
    1+2D\Lambda.
\end{equation}
Recall that \(\chi(t)=e^t/(1+e^t)\). Since
\(\mu_{H,u}(X_e=1)=\chi(\log R_e(H,u))\) and
\(\|\chi'\|_\infty\le1/4\), \eqref{eq:matching-log-odds-gradient} gives
\begin{equation}\label{eq:matching-probability-gradient}
    \sum_{f\in E(H)}
    \left|
        \frac{\partial}{\partial u_f}\mu_{H,u}(X_e=1)
    \right|
    \le
    \frac14(1+2D\Lambda).
\end{equation}
Since
\[
\mu_{H,u}(X_e=1)=\E_{\mu_{H,u}}[X_e]
=
\frac{1}{\mathcal Z_{H,u}}
\sum_{M \in\Omega_H}
x_e(M)
\exp\left(\sum_{g\in E(H)}u_g\, x_g(M)\right),
\]
differentiating with respect to \(u_f\) yields
\begin{align}
\nonumber  \frac{\partial}{\partial u_f}\mu_{H,u}(X_e=1)
&=
\E_{\mu_{H,u}}[X_eX_f]
-
\E_{\mu_{H,u}}[X_e]\,
\frac{\partial}{\partial u_f}\log \mathcal Z_{H,u} \\
\nonumber &=
\E_{\mu_{H,u}}[X_eX_f]
-
\E_{\mu_{H,u}}[X_e]\E_{\mu_{H,u}}[X_f] \\
\label{eq:matching-covariance-derivative} &=
\Cov_{\mu_{H,u}}(X_e,X_f).
\end{align}
Combining \eqref{eq:matching-probability-gradient} and
\eqref{eq:matching-covariance-derivative} proves
\eqref{eq:matching-covariance-row-sum}.

Let \(C\) be the covariance matrix of \((X_f)_{f\in E(H)}\), and let
\(\|C\|_{2\to2}\) denote its Euclidean operator norm. Since \(C\) is symmetric,
\eqref{eq:matching-covariance-row-sum} implies
\[
    \theta^{\mathsf T}C\theta
    \le
    \|C\|_{2\to2}\|\theta\|_2^2
    \le
    \left(\max_e\sum_f|C_{ef}|\right)\|\theta\|_2^2
    \le
    C_{\mathrm{cov}}(D,\Lambda)\|\theta\|_2^2,
\]
which is \eqref{eq:matching-linear-variance}.
\end{proof}

\subsection{Perturbing edge weights}

\begin{corollary}\label{cor:matching-perturbation}
Let \(\cG=(\cV,\cE)\) be a finite graph with \(\Delta(\cG)\le D\), let
\(\Omega_\cG\) be its set of matchings, and let \(\mu_w\) and \(\mu_{w'}\) be
the matching measures on \(\Omega_\cG\) with edge weights \(w\) and \(w'\),
respectively.
Suppose that
\[
    \max\{|w_e|,|w'_e|\}\le W',
    \qquad e\in\cE.
\]
Then there exists \(C_{\mathrm{match}}=C_{\mathrm{match}}(D,W')<\infty\) such that
\begin{equation}\label{eq:matching-perturbation}
    \dH^2(\mu_{w'},\mu_w)
    \le
    C_{\mathrm{match}}\sum_{e\in\cE}(w'_e-w_e)^2.
\end{equation}
\end{corollary}

\begin{proof}
Set
\[
    \theta_e=w'_e-w_e,
    \qquad
    T_\theta(M)=\sum_{e\in\cE}\theta_ex_e(M).
\]
For \(t\in[0,1]\), let \(w_t=w+t\theta\) and
\(A(t)=\log\mathcal Z_{\cG,w_t}\). Differentiating twice gives
\begin{equation}\label{eq:matching-log-partition-variance}
    A''(t)=\Var_{\mu_{w_t}}(T_\theta).
\end{equation}
Since \(|w_{t,e}|\le W'\), equations
\eqref{eq:matching-log-partition-variance} and
\eqref{eq:matching-linear-variance}, with \(\Lambda=e^{W'}\), give
\begin{equation}\label{eq:matching-log-partition-bound}
    A''(t)
    \le
    C_{\mathrm{cov}}(D,e^{W'})
    \sum_{e\in\cE}\theta_e^2.
\end{equation}
Since \(A'(1)=\E_{\mu_{w'}}[T_\theta]\), the definition of KL divergence gives
\begin{equation}\label{eq:matching-kl-path-integral}
    \dKL(\mu_{w'},\mu_w)
    =
    A'(1)-A(1)+A(0)
    =
    \int_0^1 tA''(t)\,dt.
\end{equation}
Substituting \eqref{eq:matching-log-partition-bound} into
\eqref{eq:matching-kl-path-integral} gives
\[
    \dKL(\mu_{w'},\mu_w)
    \le
    \frac12C_{\mathrm{cov}}(D,e^{W'})
    \sum_{e\in\cE}(w'_e-w_e)^2.
\]
Inequality~\eqref{eq:H-vs-KL} proves \eqref{eq:matching-perturbation}, with
\(C_{\mathrm{match}}=\frac12C_{\mathrm{cov}}(D,e^{W'})\).
\end{proof}

\subsection{Random perturbations of the edge weights}

\begin{lemma}\label{lem:random-weight-logit}
Let \(\widetilde p_e\) and \(\widetilde w_e\) be defined by
\eqref{eq:matching-tilde-weights}. Assume that
Assumptions~\ref{ass:edge-weight} and~\ref{ass:noise-4.3} hold. Then there exist
\(c_W>0\) and \(C_W<\infty\) such that, if
\(s_e^2\le c_W\), then
\begin{equation}\label{eq:random-weight-logit}
    |\widetilde w_e-w_e|
    \le
    C_W s_e^2.
\end{equation}
\end{lemma}

\begin{proof}
Set
\[
    p_e=\chi(w_e),
    \qquad
    \widetilde p_e=\E[\chi(w_e+\xi_e)],
    \qquad
    a_W=\chi(-W).
\]
Then \(a_W\le p_e\le1-a_W\). Taylor's theorem and
\(\E[\xi_e]=0\) give
\[
    |\widetilde p_e-p_e|
    \le
    \frac12\|\chi''\|_\infty s_e^2.
\]
Choose \(c_W>0\) so that
\[
    \frac12\|\chi''\|_\infty c_W\le\frac{a_W}{2}.
\]
If \(s_e^2\le c_W\), then
\[
    \frac{a_W}{2}
    \le
    \widetilde p_e
    \le
    1-\frac{a_W}{2}.
\]
Apply the mean value theorem to \(\logit\) on this interval:
\[
    |\widetilde w_e-w_e|
    \le
    \frac{|\widetilde p_e-p_e|}
    {(a_W/2)(1-a_W/2)}
    \le
    C_Ws_e^2.
\]
\end{proof}

\begin{proof}[Proof of Theorem~\ref{thm:noisy-matching}]
By \eqref{eq:noisy-matching-kernel-identity} and
\eqref{eq:matching-modified-posterior}, the invariant distribution of the noisy chain
is \(\mu_{\widetilde w_N}\).

Let \(\purple{c_W},C_W\) be the constants in Lemma~\ref{lem:random-weight-logit}. Decrease
\(c_W\), if necessary, so that \(c_WC_W\le1\). If
\[
    \max_{e\in\cE_N}s_{N,e}^2\le c_W,
\]
then
\[
    |\widetilde w_{N,e}-w_{N,e}|
    \le
    C_Ws_{N,e}^2,
    \qquad e\in\cE_N,
\]
and \(|\widetilde w_{N,e}|\le W+1\). Apply
Corollary~\ref{cor:matching-perturbation} with \(W'=W+1\):
\[
    \dH^2(\mu_{\widetilde w_N},\mu_{w_N})
    \le
    C_{\mathrm{match}}C_W^2
    \sum_{e\in\cE_N}s_{N,e}^4,
\]
where \(C_{\mathrm{match}}=C_{\mathrm{match}}(D,W+1)\). This proves
\eqref{eq:noisy-matching-main-bound}.

If \(s_{N,e}\le s_N\) for all \(e\in\cE_N\), then
\[
    \dH^2(\mu_{\widetilde w_N},\mu_{w_N})
    \le
    C_{D,W}N_{\cE}s_N^4.
\]
The quantity \(C_{D,W}N_{\cE}s_N^4\) tends to zero when \(s_N=o(N_{\cE}^{-1/4})\).
\end{proof}

\end{document}